\documentclass{amsart}
\usepackage{verbatim}
\usepackage{epsfig}
\usepackage{graphicx}
\usepackage{mathrsfs}
\usepackage{amsmath, amssymb, amsfonts, euscript, enumerate}
\usepackage{amsthm}
\usepackage{txfonts}
\usepackage{mathtools}
\usepackage{color,wrapfig}
\usepackage{subfigure}

\def\bR{\mathbb{R}}

\def\cA{\mathcal{A}}
\def\cB{\mathcal{B}}

\def\cH{\mathcal{H}}

\def\sL{\mathscr{L}}

\newcommand{\R}{{\mathbb R}}
\newcommand{\N}{{\mathbb N}}
\newcommand{\Z}{{\mathbb Z}}

\newcommand{\e}{\epsilon}
\newcommand{\ve}{\varepsilon}
\newcommand{\al}{\alpha}

\newcommand{\p}{\partial}
\newcommand{\vp}{\varphi}

\newcommand{\supp}{\operatorname{supp}}
\newcommand{\Cut}{\operatorname{Cut}}
\newcommand{\Jac}{\operatorname{Jac}}
\newcommand{\tr}{\operatorname{tr}}
\newcommand{\D}{\nabla}
\newcommand{\ra}{\rightarrow}
\newcommand{\La}{\triangle}
\newcommand{\bs}{\backslash}

\newcommand{\vol}{\operatorname{Vol}}

\providecommand{\ip}[1]{\left\langle#1\right\rangle}
\providecommand{\set}[1]{\{#1\}}

\providecommand{\abs}[1]{\lvert#1\rvert}

\newtheorem{thm}{Theorem}[section]
\newtheorem{lemma}[thm]{Lemma}
\newtheorem{cor}[thm]{Corollary}

\newtheorem{prop}[thm]{Proposition}
\newtheorem{definition}[thm]{Definition}

\newtheorem*{notation}{Notation}

\theoremstyle{definition}
\newtheorem*{acknowledgment}{Acknowledgment}

\begin{document}
\title[Harnack inequality for nondivergent parabolic operators]{Harnack inequality for nondivergent parabolic operators on Riemannian manifolds}
\author{Seick Kim}
\address{Seick Kim:\\
Department of Computational Science and Engineering,
Yonsei University,
50 Yonsei-ro, Seodaemun-gu, Seoul, 120-749, Republic of Korea}
\email{kimseick@yonsei.ac.kr}

\author{Soojung Kim}
\address{Soojung Kim:\\
School of Mathematical Sciences,
Seoul National University,
1 Gwanak-ro, Gwanak-gu,  Seoul, 151-747, Republic of Korea
}
\email{soojung26@gmail.com}
\author{Ki-ahm Lee}
\address{Ki-Ahm Lee:\\
School of Mathematical Sciences,
Seoul National University,
1 Gwanak-ro, Gwanak-gu,  Seoul, 151-747, Republic of Korea
}
\email{kiahm@math.snu.ac.kr}

\maketitle
\begin{abstract}
We consider second-order linear parabolic operators in non-divergence form that are intrinsically defined on Riemannian manifolds.
In the elliptic case, Cabr\'e proved a global Krylov-Safonov Harnack inequality under the assumption that the sectional curvature of the underlying manifold is nonnegative.
Later, Kim improved Cabr\'e's result by replacing the curvature condition by a certain condition on the distance function.
Assuming essentially the same condition introduced by Kim, we establish Krylov-Safonov Harnack inequality for nonnegative solutions of the non-divergent parabolic equation.
This, in particular, gives a new proof for Li-Yau Harnack inequality for positive solutions to the heat equation in a manifold with nonnegative Ricci curvature.
\end{abstract}

\tableofcontents 
\newpage

\section{Introduction and main results}\label{sec-intro}
In this paper, we study Harnack inequalities for solutions of second-order parabolic equations of non-divergence type on Riemannian manifolds.
Let $(M,g)$ be a smooth, complete Riemannian manifold of dimension $n$.
For $x\in M$ and $t \in \bR$, let $A_{x,t}$ be a positive definite symmetric endomorphism of $T_x M$, where $T_x M$ is the tangent space of $M$ at $x$.
We denote $\ip{X,Y}:= g(X,Y)$ and $\abs{X}^2:=\ip{X,X}$ and assume that 
\begin{equation}
\lambda \abs{X}^2 \le \ip{A_{x,t}\, X,X} \le \Lambda |X|^2,\quad \forall (x,t)\in M\times \bR,\quad \forall X\in T_x M
\end{equation} 
for some positive constants $\lambda$ and $\Lambda$.
We consider a second-order, linear, uniformly parabolic operator $\sL$ defined by
\begin{equation}
\sL u=L u - u_t:= \mathrm{trace} (A_{x,t}\circ  D^2u)-u_t \quad\mbox{in}\,\,\,M\times \bR,
\end{equation}
where $\circ$ denotes composition of endomorphisms and $D^2 u$ denotes the Hessian of the function $u$ defined by
\[
D^2 u \cdot X=\nabla_X \nabla u,
\]
where $\nabla u(x) \in T_x M$ is the gradient of $u$ at $x$.
Notice that in the special case when $A_{x,t}\equiv \mathrm{Id}$,  the equation $\sL u=0$ simply becomes the usual heat equation $u_t-\Delta u=0$.

In the elliptic setting, Cabr\'e proved in a remarkable paper \cite{Ca} that if the underlying manifold $M$ has nonnegative sectional curvature, then Krylov-Safonov type (elliptic) Harnack inequality holds for solutions of uniformly elliptic equations in non-divergence form.
Later, Kim \cite{K} improved Cabr\'e's result by removing the sectional curvature assumption and imposing a certain condition on distance function which, in the parabolic setting, should read  as follows:
For all $p \in M$, we have
\begin{align}\label{cond-M-1}
\La d_p(x) &\leq \frac{n-1}{d_p(x)}\quad\mbox{for}\quad x\notin\Cut(p)\cup\{p\},\\
\label{cond-M-2}
L d_p(x) &\leq \frac{a_{L}}{d_p(x)}\quad\mbox{for}\quad x\notin\Cut(p)\cup\{p\},\quad t\in \bR,
\end{align}
where $d_p(x)=d(p,x)$ is the geodesic distance between $p$ and $x$, $\Cut(p)$ denotes the cut locus of $p$, and $a_L$ is some positive constant that is fixed by the operator $L$.
We shall prove that if the above conditions \eqref{cond-M-1} and \eqref{cond-M-2} hold, then we have Krylov-Safonov Harnack inequality for the parabolic operator $\sL$; i.e., if $u$ is a (smooth) nonnegative solution of $\sL u =f$ in a cylinder $K_{2R}:=B_{2R}(x_0)\times (t_0-4R^2, t_0)$, where $x_0\in M$ and $t_0\in \bR$, then we have
\begin{equation}
\label{eq:I-05}
\sup_{K_R^-} u \le C\left\{ \inf_{K_R^+}u+ R^2 \left(\frac{1}{\vol(K_{2R})}\int_{K_{2R}} \abs{f}^{n+1} \right)^{\frac{1}{n+1}}\right\},
\end{equation}
where $K_R^-:=B_R(x_0) \times(t_0-3R^2,t_0-2R^2)$, $K_R^+:=B_R(x_0) \times(t_0-R^2,t_0)$, $\vol$ denotes the volume, and $C$ is a uniform constant depending only on $n,\lambda,\Lambda$ and $a_{L}$.
It is well known that the condition \eqref{cond-M-1} holds if the manifold $M$ has nonnegative Ricci curvature.
Also, as it is proved in \cite{K}, the condition \eqref{cond-M-2} is satisfied, for example, if for all $x \in M$ and any unit vector $e\in T_x M$, we have $\mathcal{M}^-[R(e)] \ge 0$.
Here,  $R(e)$ is the Ricci transformation of $T_x M$ into itself given by $R(e)X:=R(X,e)e$, where $R(X,Y)Z$ is the Riemannian curvature tensor, and
\[
\mathcal{M}^-[R(e)]=\mathcal{M}^-[R(e),\lambda, \Lambda]:=\lambda \sum_{\kappa_i>0} \kappa_i + \Lambda \sum_{\kappa_i <0} \kappa_i,
\]
where $\kappa_i$ are eigenvalues of the (symmetric) endomorphism $R(e)$.
In the case when $\sL$ is the heat operator and $M$ has nonnegative Ricci curvature, then the  condition $\mathcal{M}^-[R(e)]\ge 0$ is satisfied and thus the Harnack inequality \eqref{eq:I-05} holds; i.e., if $M$ has nonnegative Ricci curvature, then we have
\[
\sup_{K_R^-} u \le C_n\left\{ \inf_{K_R^+}u+ R^2 \left(\frac{1}{\vol(K_{2R})}\int_{K_{2R}} \abs{u_t-\Delta u}^{n+1} \right)^{\frac{1}{n+1}}\right\},
\]
where $C_n$ is a constant that depends only on the dimension $n$.
This, in particular implies the Harnack inequality of Li and Yau \cite{LY}.
Also, in the case when $M$ has nonnegative sectional curvature, then the condition $\mathcal{M}^-[R(e)]\ge 0$ is trivially satisfied and we have the inequality \eqref{eq:I-05} with a constant $C$ depending only on $n, \lambda, \Lambda$, which especially reproduces the Harnack inequality by Krylov and Safonov \cite{KS} in the Euclidean space.

One crucial ingredient in proving the Euclidean Krylov-Safonov Harnack inequality is the Krylov-Tso estimate, which is the parabolic counterpart of the Aleksandrov-Bakelman-Pucci (ABP) estimate.
The Krylov-Tso estimate as well as the classical ABP estimate is proved using affine functions, which have no intrinsic interpretation in  general Riemannian manifolds.
In the elliptic case, Cabr\'e ingeniously overcame this difficulty by replacing the affine functions by quadratic functions; quadratic functions have geometric meaning as the square of distance functions.
Following Cabr\'e's approach, we introduce an intrinsically geometric version of Krylov-Tso normal map, namely,
\[
\Phi(x,t):=\left(\exp_x \nabla_x u(x,t), -\frac{1}{2}d\left(x,\exp_x \nabla u(x,t)\right)^2 - u(x,t)\right).
\]
The map $\Phi$ is called  the parabolic normal  map related to $u(x,t)$.
A few remarks are in order regarding the normal map.
In  the classical ABP (and Krylov-Tso) estimate, an affine function concerning with the (elliptic) normal map $x\mapsto \nabla u(x)$  plays a role to bound the maximum of $u$ by  estimating the measure of the image of  the normal map.
Since an affine function cannot be generalized naturally to an intrinsic object in Riemannian manifolds, Cabr\'e used paraboloids instead in \cite{Ca}.
The map
\[
p\mapsto \min_{\Omega} \set{u(x)-p\cdot x}\quad\mbox{for a domain }\,\Omega
\]
is considered (up to a sign) as the Legendre transform of $u$. 
Krylov \cite{Kr} discovered the parabolic version of the
Aleksandrov-Bakelman maximum principle and Tso \cite{T} later simplified his proof by using the map
\[
(x,t)\mapsto (\nabla_x u(x,t), \nabla_xu(x,t)  \cdot x -u(x,t)).
\]

We end the introduction by stating our main theorems.
The rest of the paper shall be devoted to their proof.
Below and hereafter, we denote
\[
\fint_Q f:=\frac{1}{\vol(Q)} \int_Q f
\]
and
\[
K_r(x_0,t_0):=B_r(x_0)\times (t_0-r^2,t_0],\quad (x_0, t_0) \in M \times \bR.
\]

\begin{thm}[Harnack inequality]\label{thm-harnack-I}
Suppose conditions \eqref{cond-M-1}, \eqref{cond-M-2} hold.
Let $u$ be a nonnegative smooth function in $K_{2R}(x_0,4R^2)$, where $x_0\in M$ and $R>0$.
Then, we have
\begin{equation}
\sup_{K_R(x_0,2R^2)}u\leq C\left\{ \inf_{K_R(x_0,4R^2)}u+R^2 \left(\fint_{K_{2R}(x_0,4R^2)} \abs{\sL u}^{n+1} \right)^{\frac{1}{n+1}}\right\},
\end{equation} 
where $C$ is a uniform constant depending only on $n,\lambda,\Lambda$ and $a_{L}$.
\end{thm}

\begin{thm}[Weak Harnack inequality]\label{lem-weak-harnack-I}
Suppose the conditions \eqref{cond-M-1}, \eqref{cond-M-2} hold.
Let $u$ be a nonnegative smooth function satisfying $\sL u \le f$ in $K_{2R} (x_0,4R^2)$, where $x_0\in M$ and $R>0$. 
Then, we have
\begin{equation*}
\left(\fint_{K_{R}(x_0,2R^2)} u^p \right)^{\frac{1}{p}}\leq C\left\{ \inf_{K_{R}(x_0,4R^2)}u+ R^2 \left(\fint_{K_{2R}(x_0,4R^2)} \abs{f^+}^{n+1} \right)^{\frac{1}{n+1}}\right\};\quad f^+:=\max(f,0),
\end{equation*} 
where $p\in (0,1)$ and $C$ are uniform constants depending only on $n,\lambda,\Lambda$ and $a_L$.
\end{thm}

\section{Preliminaries}\label{sec-pre}

Let  $(M, g)$ be a smooth, complete Riemannian manifold of dimension $n$,  where $g$ is the Riemannian metric  and  $\vol:=\vol_g$ is  the reference measure on $M$. We denote $\langle X,Y\rangle:=g(X,Y)$  and $|X|^2:=\langle X,X\rangle$ for $X,Y\in T_xM$,  where $T_x M$ is the tangent space at $x\in M$. Let $d(\cdot,\cdot)$ be the distance function on $M$.  For a given point $y\in M$,   $d_y(x)$  denotes  the distance function from $y$,  i.e., $d_y(x):=d(x,y)$.  

We recall  the exponential map $\exp: TM \to  M$.    If $\gamma_{x,X}:\R\to M$ is the geodesic        starting from $x\in M$ with velocity $X\in T_xM$,  then the exponential map  is defined by
$$\exp_x(X):=\gamma_{x,X}(1). $$
We note that the geodesic $\gamma_{x,X}$ is defined for all time since $M$ is complete.  Given two points $x,y \in M$,   there exists a unique minimizing geodesic $\exp_x(tX)$  joining  $x$ to $y$ with $y= \exp_x(X)$ and we will write $X= \exp_x^{-1}(y)$.  

 For $X\in T_xM$ with $|X|=1$,   we define $$t_c(X) := \sup\left\{ t >0 :  \exp_x(sX) \,\,\,\mbox{is minimizing between}\,\,\, x \,\,\,\mbox{and $\exp_x(t X) $}\right\}. $$   If $t_c(X)<+\infty$,  $\exp_x(t_c(X) X)$ is a cut point of $x$.   The cut locus of $x$ is defined as the set of all cut points of $x$,  that is, 
$$ \Cut(x):= \left\{\exp_x(t_c(X) X)  : X\in T_x M \,\,\,\mbox{with $|X|=1,\,\, t_c(X)<+\infty$}\right\}. $$  Define 
$$E_x := \left\{t X \in T_xM :  0\leq t<t_c(X),\,\, X\in T_x M \,\,\,\mbox{with $|X|=1$}\right\}\subset T_xM. $$ 
One can show that  
 $\Cut(x)= \exp_x(\p E_x), M=\exp_x(E_x)\cup \Cut(x)$ and $\exp_x : E_x \to \exp_x(E_x)$ is a diffeomorphism.  We note that $\Cut(x)$ is closed and has measure zero.  For any $x\notin \Cut(y)$ with $x\neq y$,   then $d_y$ is smooth at $x$ and the Gauss lemma implies that
$$\D d_y(x)=-\frac{\exp_x^{-1} (y)}{|\exp_x^{-1} (y)|}$$
and $$\D( d^2_y/2)(x)=-\exp_x^{-1} (y).$$

Let the Riemannain curvature tensor be defined by  
$$ R(X, Y)Z = \D_X \D_YZ- \D_Y\D_XZ  -\D_{ [X,Y ]}Z, $$
where $\D$ stands for the Levi-Civita connection. For a unit vector $e\in T_xM$,   $R(e)$ will denote the Ricci  transform of $T_xM$ into itself given by $R(e)X :=R(X, e)e$.

For  $u\in C^{\infty}(M)$,  the Hessian operator  $ D^2u(x) : T_xM \to T_xM$ is  defined by 
$$D^2u(x)\cdot X = \D_X \D u(x). $$ 

Let $M$ and $N$ be   Riemannian manifolds of dimension $n$ and  $\phi:M\to N$ be smooth.  
The Jacobian of $\phi$ is  the absolute value of determinant of the differential $d \phi$,  i.e., $$\Jac\phi(x):=|\det d\phi(x)|\quad\mbox{for $x\in M$}. $$  
 We  quote the following lemma from  Lemma 3.2 in \cite{Ca},  in which the Jacobian of the  map $x \mapsto\exp_x (\D v(x))$ is   computed explicitly.  
 \begin{lemma}[Cabr\'e]\label{lem-Ca-jac}
 Let $v$ be a smooth function in an open set $\Omega$ of $M$. Define  the map $\phi:\Omega\to M  $ by 
  $$\phi(x):=\exp_{x}\D v(x). $$
Let $x\in\Omega$ and   suppose that $\D v(x)\in E_x$. Set  $y:=\phi(x)$. Then we have 
\begin{align*}
\Jac \phi(x)
 =\Jac \exp_{x}(\D v(x)) \cdot \Big| \det D^2 \left(v+d^2_y/2\right) \left(x\right)\Big|,
\end{align*}
where    $\Jac\exp_x(\D v(x))$ denotes the Jacobian of $\exp_x$, a map from $T_xM$ to $M$,  at the point $\D v(x) \in T_xM$.  
\end{lemma}

Under the condition \eqref{cond-M-1}, we have the estimate for Jacobian of the exponential map and   Bishop's  volume comparison theorem as follows. We state the known results as a lemma.  The proof can be found  in  \cite[p. 286]{K} (see also \cite{L}).
\begin{lemma} \label{lem-jac-bishop}
Suppose that $M$ satisfies \eqref{cond-M-1}.  
\begin{enumerate}[(i)]
\item For any $x\in M$ and $X\in E_x, $$$\Jac\exp_x(X)= |\det d\exp_x(X)|\leq 1. $$
\item  (Bishop) For any $x\in M$,  $\vol(B_{R}(x))/R^{n}$ is nonincreasing with respect to $R$,  where $B_R(x)$ is a geodesic ball of radius $R$ centered at $x$.  Namely,
\begin{equation*}
\frac{\vol(B_R(x))}{\vol(B_r(x))}\leq\frac{R^{n}}{r^{n}}\quad\text{if $\,0<r<R$.}
\end{equation*}
In particular, $M $ satisfies the volume doubling property; i.e., $ \vol(B_{2R}(x))\leq 2^n\vol(B_R(x))$.
\end{enumerate}
\end{lemma}
 

The following is the area formula, which follows easily from  the area formula in Euclidean space and  a partition of unity. 
\begin{lemma}[Area formula]
 For any smooth function $\phi:M\times \R\to M\times \R$ and any measurable set $E\subset M\times\R$, we have
\begin{equation*}
\int_E \Jac\phi (x,t)dV(x,t)=\int_{M\times\R} \cH^0[E\cap\phi^{-1}(y,s)] dV(y,s),\end{equation*}
where $\cH^0$ is the counting measure.
\end{lemma}

\begin{notation}
Let us summarize  the notations and definitions that  will be used.
\begin{itemize}
\item Let $r>0, \rho>0$,  $z_o\in M$ and $t_o\in\R$.  We denote   
$$K_{r,\,\rho}(z_o,t_o):=B_r(z_o)\times(t_o-\rho,t_o], $$  
where $B_r(z_o)$ is a geodesic ball of radius $r$ centered at $z_o$.
\item We denote  $K_{r}(z_o,t_o):= K_{r, \,r^2}(z_o,t_o)$.
\item  We say that  a constant $C$ is uniform if $C$ depends only on $n, \lambda,\Lambda$ and $a_L$.
\item We denote $\fint_Q f:=\frac{1}{\vol(Q)}\int_Q f.$
\item We denote $|Q|:=\vol(Q)$. 
\item We denote  the trace by $\tr$.
\end{itemize}
\end{notation}
\section{Key lemma}\label{sec-key-lemma}

In this section, we obtain  Aleksandrov-Bakelman-Pucci-Krylov-Tso type estimate (Lemma \ref{lem-abp-type})  for parabolic Harnack inequalities.
   We begin  with  direct computation of the Jacobian of the parabolic normal map $\Phi$ below, which is a parabolic analogue of  Lemma \ref{lem-Ca-jac}.
 
\begin{lemma}\label{lem-Jac-normal_map}
 Let $v$ be a smooth function in an open set $K$ of $M\times\R$. Define  the map $\phi:K\to M  $ by 
  $$\phi(x,t):=\exp_{x}\D_{x}v(x,t)$$
and  the map $\Phi:K\ra M\times \R$  by
$$\Phi(x,t):=\left(\phi(x,t), -\frac{1}{2}d\left(x, \phi(x,t)\right)^2-v(x,t)\right). $$
Let $(x,t)\in K$ and assume that $\D_x v(x,t)\in E_x.$ Set $y:=\phi(x,t).$ Then 
\begin{align*}
\Jac \Phi(x,t)
=\Jac \exp_{x}(\D_x v(x,t)) \cdot \Big|(-v_t)\det \left( D_x^2 \left(v+d^2_y/2\right)\right)\Big|,
\end{align*}
where $\Jac\exp_x(\D _xv(x,t))$ denotes the Jacobian of $\exp_x$ at the point $\D_x v(x,t) \in T_xM$.  
\end{lemma}
\begin{proof}
We may assume that $\D_x v(x,t)\neq 0$, which is equivalent to $x\neq y$.
Let $(\xi,\sigma)\in T_xM\times \R\bs \{(0,0)\}$ and let $\gamma=(\gamma_1,\gamma_2)$ be the geodesic  with  $\gamma (0)=(x,t)$ and $\gamma'(0)=(\xi,\sigma)$. We note that $\gamma_1(\tau)=\exp_x\tau\xi$ and $\gamma_2(\tau)=t+\sigma\tau$.  
Set $$Y(s,\tau):=\exp_{\gamma_1(\tau)}\left[ s\D_x v(\gamma(\tau))\right]. $$
Consider the family of geodesics   (in the parameter $s$)
$$\Pi (s,\tau):=\left(Y(s,\tau), \gamma_2(\tau)-s\left\{\frac12 d\left(\gamma_1(\tau), \phi(\gamma(\tau))\right)^2+\, v(\gamma (\tau))+\gamma_2(\tau)\right\}\right)$$
that joins $\Pi(0,\tau)=\gamma(\tau)$  to  $\Pi(1,\tau)=\Phi(\gamma (\tau))$. 
Then we define
$$J(s):=\frac{\partial}{\partial \tau}\Big|_{\tau=0}\Pi (s,\tau), $$
which is a { Jacobi field} along $$X(s):=\left(\exp_x(s\D_x v(x,t)),t-{\color{black}s}\left\{\frac12{ d\left(x,\phi(x,t)\right)^2}+ \,v(x,t)+t\right\}\right). $$
Simple computation says that
$${ J(0)=(\xi,  \sigma)}\quad\mbox{and}\quad J(1)=\frac{\partial}{\partial\tau}\Big|_{\tau=0}\Phi(\gamma(\tau))=d\Phi{(x,t)}\cdot(\xi,\sigma). $$
We also have 
\begin{align*}
D_sJ(0)&=\left(D_x^2v(x,t)\,\xi+\sigma\D_xv_t(x,t)\right., \quad-\sigma v_t(x,t)-\sigma\\
&\left.-\left\langle\D_x\left(d_x^2/2\right)(y),\,d\exp_x(\D_xv(x,t))\cdot\left(D^2_x\left(v+d_y^2/2\right)(x,t)\cdot\xi+\sigma\D_xv_t(x,t)\right)\right\rangle\right).
\end{align*}
In fact,  we have
\begin{align*}
&D_sJ(0)=D_s\frac{\partial\Pi}{\partial\tau}\Big|_{s=0,\tau=0} =D_\tau\frac{\partial\Pi}{\partial s}\Big|_{s=0,\tau=0}\\&=D_\tau\Big|_{\tau=0}\left(\D_xv(\gamma(\tau)),-\frac{1}{2}d(\gamma_1(\tau),\phi(\gamma(\tau)))^2-v(\gamma(\tau))-\gamma_2(\tau)\right)\\
&=\left(D^2_xv(x,t) \cdot\xi+\sigma\D_xv_t(x,t), \right.\\
&\,\,-\left\langle\D_x(d_{y}^2/2)(x),\,\xi\right\rangle-\left.\left\langle\D_x(d_{x}^2/2)(\phi(x,t)),\frac{\partial}{\partial\tau}\phi(\gamma(\tau))\Big|_{\tau=0} \right\rangle
-\left\langle\D_xv(x,t) ,\,\xi\right\rangle-\sigma v_t-\sigma\right)\\
&=\left(D^2_xv\cdot\xi+\sigma\D_xv_t,-\left\langle\D_x(d_{x}^2/2)(\phi(x,t)),\frac{\partial}{\partial\tau}\phi(\gamma(\tau))\Big|_{\tau=0} \right\rangle-\sigma v_t-\sigma\right),
\end{align*}  since  $\D_x(d^2_y/2)(x)= -\exp^{-1}_x(y)=-\D_xv(x,t)$. Then we use Lemma \ref{lem-Ca-jac}  to obtain
$$ \frac{\partial}{\partial\tau}\phi(\gamma(\tau))\Big|_{\tau=0}=d\exp_x(\D_xv(x,t))\cdot\left(D_x^2\left(v+ d_{y}^2/2\right)(x,t)\cdot\xi+\sigma\D_xv_t(x,t)\right). $$

On the other hand, consider the Jacobi field $J_{\xi,\sigma}$ along $X(s)$ satisfying 
$$J_{\xi,\sigma}(0)=(\xi,\sigma)\quad\mbox{and}\quad J_{\xi,\sigma}(1)=(0,0). $$ Then we can check that
 $$ J_{\xi,\sigma}(s)=\frac{\partial}{\partial\tau}\Psi\Big|_{\tau=0}\quad \mbox{and}\quad D_sJ_{\xi,\sigma}(0)=\left(-D^2_x\left( d_{y}^2/2\right)(x)\cdot\xi,-\sigma\right), $$where 
 $$\Psi(s,\tau):=\left(\exp_{\gamma_1(\tau)}s\exp^{-1}_{\gamma_1(\tau)}\phi(x,t),\gamma_2(\tau)-s\left\{\frac{1}{2}d(x,\phi(x,t))^2+v(x,t)+\gamma_2(\tau)\right\}\right).$$  (We refer \cite[Lemma 3.2]{Ca}  for the proof.)

Define $\tilde J_{\xi,\sigma}:=J-J_{\xi,\sigma}$. The Jacobi field $\tilde J_{\xi,\sigma}$ along $X(s)$ satisfying
$$\tilde J_{\xi,\sigma}(0)=(0,0)\quad\mbox{and}\quad D_s\tilde J_{\xi,\sigma}(0)=D_sJ(0)-D_sJ_{\xi,\sigma}(0)$$ is  written by
$$d\exp_{(x,t)}(sX'(0))\cdot\left(sD_s\tilde J_{\xi,\sigma}(0)\right). $$
Therefore, 
we have
\begin{align*}
J(1)=\tilde J_{\xi,\sigma}(1)=  d \exp_{(x,t)}\left(\D_x v(x,t),-\frac{1}{2}d(x,y)^2-v(x,t)-t\right)\cdot\left(D_sJ(0)-D_sJ_{\xi,\sigma}(0)\right), 
\end{align*}
which means 
\begin{align*} 
&d\Phi(x,t) \cdot(\xi,\sigma)=d\exp_{(x,t)}\left(\D_x v(x,t),-\frac{1}{2}d(x,y)^2-v(x,t)-t\right)\cdot
\\&\left(D_x^2\left(v+ d_y^2/2\right)(x,t)\cdot\xi+\sigma \D_x v_t(x,t)\,,\right.\\& \left.-\sigma v_t -\left\langle\D_x\left(d_x^2/2\right)(y),\,d\exp_x(\D_xv(x,t))\cdot\left(D^2_x\left(v+ d_y^2/2\right)(x,t)\cdot\xi+\sigma\D_xv_t(x,t)\right)\right\rangle\right)\\
&=\left( d\exp_{x}\left(\D_x v(x,t) \right)\cdot \left(D_x^2\left(v+d_y^2/2\right)(x,t)\cdot\xi+\sigma \D_x v_t(x,t)\right)\,,\right.\\& \left.-\sigma v_t -\left\langle\D_x\left(d_x^2/2\right)(y),\,d\exp_x(\D_xv(x,t))\cdot\left(D^2_x\left(v+ d_y^2/2\right)(x,t)\cdot\xi+\sigma\D_xv_t(x,t)\right)\right\rangle\right).
\end{align*} 
To calculate the Jacobian of $\Phi$,  we introduce an orthonormal basis $\{e_1, \cdots, e_n\}$ of $T_xM$ and an orthonormal basis $\{\overline e_1,\cdots,\overline e_n\}$ of $T_yM=T_{\exp_x\D v(x,t)}M$. By setting for $i,j=1,\cdots,n$, 
\begin{align*}
A_{ij}&:=\left\langle \overline e_i, \,d\exp_{x}\left(\D_x v(x,t) \right)\cdot \left(D_x^2\left(v+ d_y^2/2\right)(x,t)e_j\right)\right\rangle,\\
b_i&:=\left\langle \overline e_i, \,d\exp_{x}\left(\D_x v(x,t) \right) \cdot\D v_t(x,t) \right\rangle,\\
c_i&:=\left\langle \overline e_i, \,\D_x\left(d_x^2/2\right)(y)\right\rangle,
\end{align*} the Jacobian matrix of  $\Phi$ at $(x,t)$   is 
\begin{equation*}\displaystyle
 \left(
\begin{array}{cc}
 A_{ij} &   b_i\\
-c_{k}A_{kj} & -v_t-b_kc_k
\end{array}
\right).
\end{equation*}
Lastly, we use  the row operations  to deduce that
\begin{equation*}\displaystyle
\Jac\Phi(x,t)=\left|\det \left(
\begin{array}{cc}
 A_{ij} &   b_i\\
0 & -v_t
\end{array}
\right)\right|=\left|(-v_t)\det (A_{ij})\right|.
\end{equation*}
This completes the proof.
\end{proof}

The following lemma will play a key role to estimate sublevel sets of $u$ in Lemma \ref{lem-decay-est-1-step}  and then to prove   a decay estimate of   the distribution function of $u$ in Lemma \ref{lem-decay-est}.
This ABP-type lemma corresponds to \cite[Lemma 4.1]{Ca}.  
\begin{lemma}\label{lem-abp-type} Suppose that $M$ satisfies the condition 
 \eqref{cond-M-2}.
Let $z_o\in M$,   $R>0$,  and $0<\eta<1$.
Let $u$ be a smooth function in $K_{\alpha_1 R,\, \alpha_2R^2}(z_o,0) \subset M\times\R$ satisfying 
\begin{equation}\label{cond-abp}
u\geq 0\quad\mbox{in}\quad K_{\alpha_1R,\, \alpha_2R^2}(z_o,0)\bs K_{\beta_1R,\,\beta_2R^2}(z_o,0) \quad\mbox{
and} \quad\inf_{K_{2R}(z_o,0)}u\leq1,\end{equation}
where $\alpha_1:=\frac{11}{\eta}$, $\alpha_2:=4+\eta^2+\frac{\eta^4}{4}$, $\beta_1:=\frac{9}{\eta}$, and $\beta_2:=4+\eta^2$.  
 Then we have 
\begin{equation}\label{eq-abp-type}
|B_R(z_o)|\cdot R^2\leq C(\eta,n,\lambda)\int_{\{u\leq M_\eta\}\cap K_{\beta_1R,\,\beta_2R^2}(z_o,0)}\left\{\left(R^2\sL u+ a_{L}+\Lambda+1 \right)^+ \right\}^{n+1}
\end{equation}
where the constant  $M_\eta>0$ depends only on $\eta>0$ and $C(\eta,n,\lambda)>0$ depends only on $\eta, n$ and $\lambda.$
\end{lemma}

\begin{proof}

\begin{figure}[b] 
  \centering
    \includegraphics[width=0.95\textwidth]{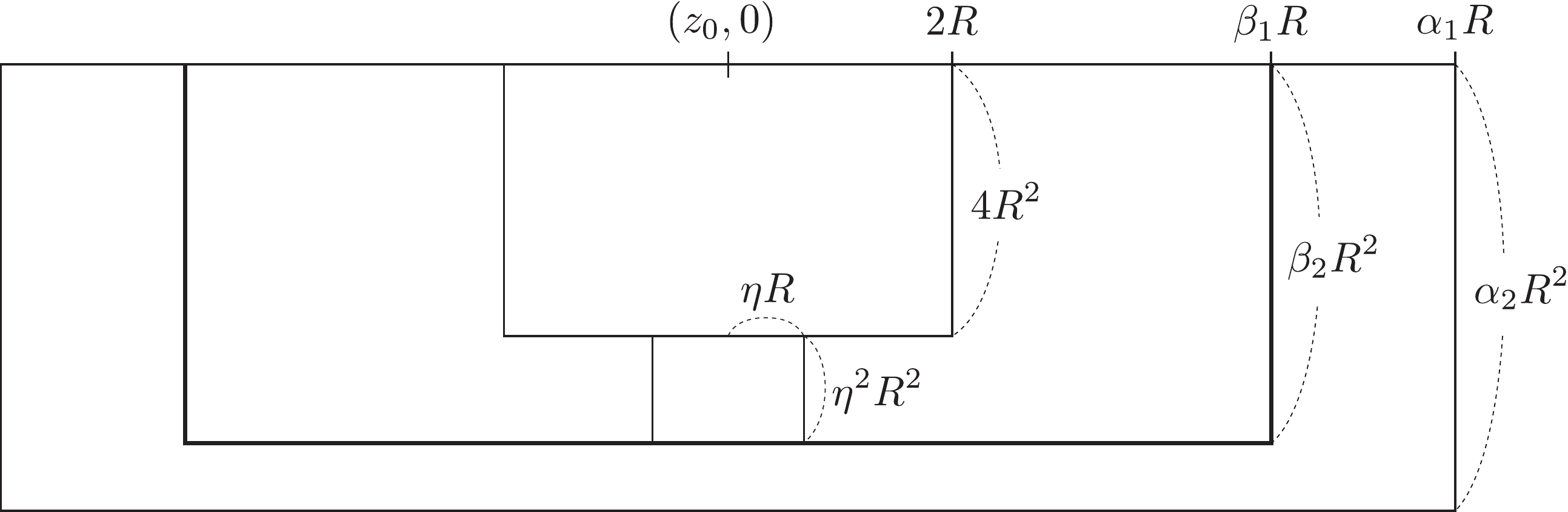}
      \caption{  $\alpha_1:=\frac{11}{\eta}, \alpha_2:=4+\eta^2+\frac{\eta^4}{4}, \beta_1:=\frac{9}{\eta}$,   $\beta_2:=4+\eta^2$.}\label{boxes}
\end{figure}
For any $\overline y\in B_R(z_o)$,  we define 
$$w_{\overline y}(x,t):=\frac{1}{2}R^2u(x,t)+\frac{1}{2}d_{\overline y}^2(x)-C_{\eta}t,\,\,\,\,C_\eta:=\frac{6}{\eta^2}.$$
From the assumption \eqref{cond-abp}, it is easy to check that 
$$\inf_{K_{2R}(z_o,0)}w_{\overline y}\leq
\left( 5+\frac{24}{\eta^2}\right)R^2=:A_\eta R^2, $$
and $$w_{\overline y}\geq\left(6+\frac{24}{\eta^2}\right)R^2=(A_\eta +1)R^2\quad\mbox{on}\,\,\,K_{\alpha_1R,\,\alpha_2R^2}(z_o,0)\bs K_{\beta_1R,\,\beta_2R^2}(z_o,0). $$
From the above observation, for any $\displaystyle (\overline y, \overline h)\in B_R(z_o)\times\left(A_\eta R^2,(A_\eta +1)R^2\right)$,  we can find a time  $\displaystyle \overline t\in\left(-\beta_2R^2,0\right)$ such that 
$$\overline h=\inf_{B_{\beta_1R}(z_o)\times(-\beta_2R^2,\overline t]}w_{\overline y}(z,\tau)=w_{\overline y}\left(\overline x,\overline t\right), $$
where the infimum is achieved at an interior point  $\overline x$ of $ B_{\beta_1R}(z_o)$. 
By the same argument as in  \cite[pp. 637-638]{Ca}, we have the following relation:
$$\overline y=\exp_x\D_x\left(\frac{1}{2}R^2u\right)\left(\overline x, \overline t\right). $$

Now, we consider the map $\Phi:K_{\alpha_1R,\,\alpha_2R^2}(z_o,0)\to M\times\R$ (with $v(x,t)=\frac{1}{2}R^2u(x,t)-C_\eta t$ in Lemma \ref{lem-Jac-normal_map}) 
 defined as
\[
\Phi(x,t):=\left(\exp_x\D_x\left(\frac{1}{2}R^2u\right)(x,t), -\frac{1}{2}d\left(x,\exp_x\D_x\left(\frac{1}{2}R^2u\right)(x,t)\right)^2-\frac{1}{2}R^2u(x,t)+C_\eta t\right).
\]
Define  a set  
\begin{align*}
E:=\left\{(x,t)\in K_{\beta_1R,\,\beta_2R^2}(z_o,0)\,:\, \exists y\in B_{R}(z_o)\quad\mbox{s.t.}\quad \displaystyle w_{y}(x,t)=\inf_{B_{\beta_1R}(z_o)\times(-\beta_2R^2,t]}w_y\leq(A_\eta+1 )R^2  \,\right\}.\end{align*}
The set $E$ is a subset of the contact set in $K_{\beta_1R,\,\beta_2R^2}(z_o,0)$ that contains a point $(x,t),$ where a concave paraboloid $-\frac{1}{2} d_{y}^2(x)+C_\eta t+C $ (for some $C$) touches $\frac{1}{2}R^2u$ from below. 
 Thus we have proved that for any $(y,s)\in B_R(z_o)\times(-(A_\eta +1)R^2,-A_\eta R^2)$,  there is at least one $(x,t)\in E$ such that
$(y,s)=\Phi(x,t)$,  namely,   $$B_R(z_o)\times\left(-(A_\eta +1)R^2,-A_\eta R^2\right)\subset \Phi(E). $$ So Area formula gives
\begin{align}\label{eq-abp-area}
|B_R(z_o)|\cdot  R^2\leq \int_{M\times\R} \cH^0\left[E\cap\Phi^{-1}(y,s)\right] dV(y,s)= \int_E\Jac \Phi (x,t)dV(x,t).
\end{align}
We notice that  for $(x,t)\in E$ and $y\in B_R(z_o)$,   $w_y(x,t)=\frac{1}{2}R^2u(x,t)+\frac{1}{2}d_y^2(x)-C_\eta t\leq (A_\eta +1)R^2$ and hence $u(x,t)\leq 2(A_\eta +1)=:M_\eta $ for $(x,t)\in E$.

Lastly, we claim that for $(x,t)\in E$, 
\begin{equation}\label{eq-abp-jac}
\Jac\Phi(x,t)\leq \frac{1}{(n+1)^{n+1}\lambda^n}\left\{\left(\frac{1}{2}R^2\sL u(x,t)+ a_{L}+\Lambda+C_\eta \right)^+ \right\}^{n+1}.\end{equation}
Fix  $(x,t)\in E$ and  $y\in B_R(z_o)$ to satisfy $$\displaystyle w_{y}(x,t)=\inf_{B_{\beta_1R}(z_o)\times(-\beta_2R^2,t]}w_y. $$ We recall  that $y=\exp_x\D_x\left(\frac{1}{2}R^2u\right)(x,t)$ (see \cite[pp. 637-638]{Ca}). 

 If $x$ is not a cut point of $y$,  then 
Lemma \ref{lem-Jac-normal_map} (with $v(x,t)=\frac{1}{2}R^2u(x,t)-C_\eta t$)  and Lemma \ref{lem-jac-bishop} (i)  imply that 
$$\Jac\Phi(x,t)\leq  \left|\left(-\frac{1}{2}R^2u_t+C_\eta\right)\det \left( D_x^2 \left(\frac{1}{2}R^2u+\frac{1}{2}d^2_y\right)\right)(x,t)\right| . $$
Since the minimum of $w_y$ in $B_{\beta_1R}(z_o)\times(-\beta_2R^2,t]$ is achieved at $(x,t)$,  we have 
$$0\leq D^2_x w_y(x,t)=D_x^2 \left(\frac{1}{2}R^2u+\frac{1}{2}d^2_y\right) \,\,\, \mbox{and }\,\,\,0\geq \partial_tw_y(x,t)=\frac{1}{2}R^2u_t-C_\eta, $$ 
where   $D^2_x w_y(x,t)\geq 0$ means that the Hessian of $w_y$ at $(x,t)$ is positive semidefinite. 
Therefore, by using the geometric and arithmetic means inequality, we get
\begin{align*}
\Jac\Phi(x,t)&\leq \left(-\frac{1}{2}R^2u_t+C_\eta\right)\det \left( D_x^2 \left(\frac{1}{2}R^2u+\frac{1}{2}d^2_y\right)\right)(x,t) \\
&\leq \frac{1}{\lambda^n}\left(-\frac{1}{2}R^2u_t+C_\eta\right)\det A_{x,t}\det \left( D_x^2\left(\frac{1}{2}R^2u+\frac{1}{2}d^2_y\right)\right) \\
&\leq \frac{1}{(n+1)^{n+1}\lambda^n}\left\{ \tr\left( A_{x,t}\circ D_x^2 \left(\frac{1}{2}R^2u+\frac{1}{2}d^2_y\right)\right)-\frac{1}{2}R^2u_t+C_\eta \right\}^{n+1} \\
&= \frac{1}{(n+1)^{n+1}\lambda^n}\left\{ \frac{1}{2}R^2\sL u(x,t) + L\left[\frac{1}{2}d^2_y\right]+C_\eta\right\}^{n+1} \\
&\leq \frac{1}{(n+1)^{n+1}\lambda^n}\left\{ \frac{1}{2}R^2\sL u+a_L+\Lambda+C_\eta\right\}^{n+1}\\
 &=\frac{1}{(n+1)^{n+1}\lambda^n}\left\{ \left(\frac{1}{2}R^2\sL u +a_L+\Lambda+C_\eta\right)^+\right\}^{n+1}, 
\end{align*} 
where we used
$$L\left[d_y^2/2\right]=d_y L d_y +\left<A_{x,t}\D d_y,\D d_y\right>\leq a_L+\Lambda|\D d_y|^2. $$

When $x$ is a cut point of $y, $  we make use of upper barrier technique due to Calabi \cite{Cal}.  
 Since $y=\exp_x\D_x \left(\frac{1}{2}R^2u\right)(x,t)$,  $x$ is not a cut point of $y_{\sigma}:=\phi_\sigma(x,t):=\exp_x\D_x\left( \frac{\sigma}{2}R^2u\right)(x,t)$ for $0\leq\sigma<1$.  Now we consider 
 $$\Phi_{\sigma}(z,\tau):=\left(\phi_\sigma(z,\tau), - \frac{\sigma}{2}R^2u(z,\tau)-\frac{1}{2}d\left(z, \phi_{\sigma}(z,\tau)\right)^2+  C_\eta \tau\right)$$
 instead of $\Phi$ since $\displaystyle\Jac\Phi(x,t)=\lim_{\sigma\uparrow1}\Jac\Phi_{\sigma}(x,t)$. As before, we have 
$$\Jac\Phi_\sigma(x,t)\leq  \left|\left(-\frac{\sigma}{2}R^2u_t+C_\eta\right)\det \left( D_x^2 \left( \frac{\sigma}{2}R^2u+\frac{1}{2}d^2_{y_\sigma}\right)\right)(x,t)\right| . $$
We note that 
\begin{align*}
&\liminf_{\sigma\uparrow1} \left|\left(-\frac{\sigma}{2}R^2u_t+C_\eta\right)\det \left( D_x^2 \left(\frac{\sigma }{2}R^2u+\frac{1}{2}d^2_{y_\sigma}\right)\right)(x,t)\right|\\ 
=&\liminf_{\sigma\uparrow1} \left|\left(- \p_t w_{y_\sigma}(x,t) \right)\det \left(  D_x^2  w_{y_\sigma}\right)(x,t)\right|
\end{align*}
for $w_{y_\sigma}(z,\tau):= \frac{1}{2}R^2u(z,\tau)+\frac{1}{2}d_{y_\sigma}^2(z)-C_\eta\tau$. 
According to the triangle inequality, we have
\begin{align*}
w_y(z,\tau)&\leq \frac{1}{2}R^2u(z,\tau)+\frac{1}{2}\left(d_{y_{\sigma}}(z)+d(y_{\sigma}, y)\right)^2-C_\eta\tau\\
&=w_{y_{\sigma}}(z,\tau)+d(y_{\sigma},y)d_{y_{\sigma}}(z)+\frac{1}{2}d(y_{\sigma},y)^2,
\end{align*}
where the equality holds at $(z,\tau)=(x,t)$. Since   $w_{y}$ has the minimum at $(x,t)$  in $B_{\beta_1R}(z_o)\times(-\beta_2R^2,t]$,   the  minimum of    $ w_{y_{\sigma}}(z,\tau)+d(y_{\sigma},y)d_{y_{\sigma}}(z) $ (in $B_{\beta_1R}(z_o)\times(-\beta_2R^2,t]$) is also    achieved at $(x,t)$,  that implies that
$$D^2_x\left(w_{y_{\sigma}}+d(y_{\sigma},y)d_{y_{\sigma}} \right)(x,t)\geq 0,\quad \p_tw_{y_{\sigma}}(x,t)\leq 0. $$ 
To bound $D^2y_{\sigma}(x)$ uniformly in $\sigma\in[1/2,1)$,  we recall  the Hessian comparison theorem (see \cite{S},\cite{SY}): Let $-k^2$ ($k>0$) be a lower bound of sectional curvature along the minimal geodesic joining $x$ and $y$.    Then for $0<\sigma<1$, 
$$D^2d_{y_\sigma} (x)\leq k \coth(kd_{y_\sigma}(x)) \mathrm{Id}$$ and hence we find a  constant $N>0$ independent of $\sigma$ such that $$D^2d_{y_\sigma}(x)\leq N \mathrm{Id}\quad\mbox{for}\,\,\,\frac{1}{2}\leq \sigma<1. $$ 
Following the above argument, for $\frac{1}{2}\leq \sigma<1$,  we obtain
\begin{align*}
0&\leq \liminf_{\sigma\uparrow1} \left(- \p_t w_{y_\sigma}(x,t) \right)\det \left(  D_x^2  w_{y_\sigma}+d(y_{\sigma},y)D^2 d_{y_{\sigma}} \right)(x,t)\\
&\leq\liminf_{\sigma\uparrow1}
\left(- \p_t w_{y_\sigma}(x,t) \right)\det \left(  D_x^2  w_{y_\sigma}+d(y_{\sigma},y)N\mathrm{Id}\right)(x,t)\\
&\leq \liminf_{\sigma\uparrow1}\frac{1}{(n+1)^{n+1}\lambda^n}\left\{ \frac{1}{2}R^2\sL u+a_L+\Lambda+C_\eta + d(y_\sigma,y)n\Lambda N \right\}^{n+1}\\
&\leq  \frac{1}{(n+1)^{n+1}\lambda^n}\left\{ \left(\frac{1}{2}R^2\sL u +a_L+\Lambda+C_\eta  \right)^+\right\}^{n+1}.
\end{align*}
Then we deduce  that  $$\Jac\Phi(x,t)\leq\frac{1}{(n+1)^{n+1}\lambda^n}\left\{\left( \frac{1}{2}R^2\sL u(x,t)+a_L+\Lambda+C_\eta\right)^+\right\}^{n+1}$$ since 
\begin{align*}\liminf_{\sigma\uparrow1} \left| \det \left(  D_x^2  w_{y_\sigma}\right)(x,t)\right|&=\liminf_{\sigma\uparrow1} \left| \det \left(  D_x^2  w_{y_\sigma}+d(y_{\sigma},y)N\mathrm{Id} \right)(x,t)\right|\\
&=\liminf_{\sigma\uparrow1}  \det \left(  D_x^2  w_{y_\sigma}+d(y_{\sigma},y)N\mathrm{Id} \right)(x,t).\end{align*} We conclude that \eqref{eq-abp-jac} is true for  $(x,t)\in E.$
Therefore the estimate  \eqref{eq-abp-type} follows from \eqref{eq-abp-area}  
since $E\subset \left\{u\leq M_\eta\right\}\cap K_{\beta_1R,\, \beta_2R^2}(z_o,0)$. 
\end{proof}

\section{Barrier functions}\label{sec-barrier}

We modify the barrier function of \cite{W}  to construct a barrier function in the Riemannian case.  First, we fix some constants that will be used frequently (see Figure \ref{boxes}); for a given $0<\eta<1$, 
$$\alpha_1:=\frac{11}{\eta},\,\alpha_2:=4+\eta^2+\frac{\eta^4}{4},\, \beta_1:=\frac{9}{\eta}\,\,\,\,\mbox{and}\,\,\,\beta_2:=4+\eta^2. $$   
\begin{lemma}\label{lem-barrier} Suppose that $M$ satisfies the condition 
 \eqref{cond-M-2}. 
 Let $z_o\in M$, $R>0$ and $0<\eta<1$. There exists a continuous function $v_{\eta}(x,t)$ in $\displaystyle K_{\alpha_1R,\,\alpha_2R^2}(z_o,\beta_2R^2) 
 $,  which is smooth in $\left(M\bs \Cut (z_o)\right)\cap K_{\alpha_1R,\,\alpha_2R^2}(z_o,\beta_2R^2)$ such that  
\begin{enumerate}[(i)]
\item $v_{\eta}(x,t)\geq 0$ in $K_{\alpha_1R,\,\alpha_2R^2}(z_o, \beta_2R^2)\,\bs\,  K_{\beta_1R,\,\beta_2R^2} (z_o,\beta_2R^2)$, 
\item $v_{\eta}(x,t)\leq 0$ in $K_{2R}(z_o,\beta_2R^2)$,
\item {  $R^2\sL v_{\eta}+a_{L}+\Lambda+1\leq 0\ $} a.e. in $K_{\beta_1R, \,\beta_2R^2}(z_o, \beta_2R^2)\bs  K_{\frac{\eta}{2} R} (z_o, \frac{\eta^2}{4}R^2), $
\item $R^2\sL  v_{\eta} \leq C_{\eta}\quad a.e. $ in $K_{\beta_1R,\, \beta_2R^2}(z_o, \beta_2R^2)$,
\item  $v_{\eta}(x,t)\geq -C_{\eta}$ in $K_{\alpha_1R,\,\alpha_2R^2}(z_o,  \beta_2R^2)$.
\end{enumerate}
Here, the constant $C_{\eta}>0$  depends only on $\eta, n,\lambda,\Lambda, a_L$ ( independent of  $R$ and $z_o$ ).
\end{lemma}
\begin{proof} Fix $0<\eta<1$.  
Consider  $$h(s,t):=-Ae^{-mt}\left(1-\frac{s}{\beta_1^2}\right)^l\frac{1}{(4\pi t)^{n/2}}\exp\left(-\alpha\frac{s}{t}\right)\quad\mbox{for $t>0$,}$$ 
as in Lemma 3.22 of  \cite{W} and   define 
  $$\psi(s,t):= h(s,t)+(a_L+\Lambda+1)t\quad\mbox{in $[0,\beta_1^2]\times[0,\beta_2]\bs [0,\frac{\eta^2}{4}]\times[0,\frac{\eta^2}{4}]$, }$$ where the positive constants $A,m,l,\alpha$ ( depending only on $\eta, n,\lambda,\Lambda, a_L$) will be chosen later. In particular, $l$ will be an odd number in $\N$.   We extend $\psi$ smoothly   in $[0, \alpha_1^2]\times[-\frac{\eta^4}{4}, \beta_2]$ to satisfy 
\begin{align*}
\psi\geq0  \quad&\mbox{on $[0, \alpha_1^2]\times[-\frac{\eta^4}{4}, \beta_2]\bs[0,\beta_1^2]\times[0, \beta_2]$, }\\
 \psi\geq-C_\eta \quad&\mbox{on $[0, \alpha_1^2]\times[-\frac{\eta^4}{4}, \beta_2]$, }
\end{align*} and 
$$\displaystyle\sup_{[0, \beta_1^2]\times[0,\beta_2]}\left\{2a_L\left|\p_s\psi\right|+\Lambda\left(2|\p_s\psi|+4s|\p_{ss}\psi|\right)+|\p_t\psi| \right\}(s,t)< C_\eta
$$
for some $C_\eta>0$.
We also assume that $\psi(s,t)$ is nondecreasing with respect to $s$ in $[0,\alpha_1^2]\times[-\frac{\eta^4}{4}, \beta_2]$.
We define 
$$v_\eta(x,t)=v(x,t):=\psi\left(\frac{d^2_{z_o}(x)}{R^2},\frac{t}{R^2}\right) \quad\mbox{for }\,\,\,\,(x,t)\in K_{\alpha_1R,\,\alpha_2R^2}(z_o, \beta_2R^2) , $$ where $d_{z_o}$ is the  distance function to $z_o$.
Properties (i) and (v) are trivial.

 We denote $ d_{z_o}(x) $ and  $h\left(\frac{d^2_{z_o}(x)}{R^2},\frac{t}{R^2}\right)$ by $d(x)$ and $\phi(x,t)$ for simplicity and we notice that for $(x,t)\in K_{\beta_1R,\,\beta_2R^2}(z_o, \beta_2R^2)\bs K_{\frac{\eta}{2}R}(z_o, \frac{\eta^2}{4}R^2)$, 
\begin{align*}
v(x,t)&= h\left(\frac{d^2(x)}{R^2}, \frac{t}{R^2}\right)+(a_L+\Lambda+1)\frac{t}{R^2}=\phi(x,t) + (a_L+\Lambda+1)\frac{t}{R^2}\end{align*}
and $\phi(x,t)$ is negative in $K_{\beta_1R,\,\beta_2R^2}(z_o, \beta_2R^2)$.

Now, we claim that 
\begin{equation}
\label{eq-barrier-phi}
\sL \phi \leq0 \quad\mbox{a.e.  in $K_{\beta_1R,\,\beta_2R^2}(z_o, \beta_2R^2)\bs K_{\frac{\eta}{2} R}(z_o,\frac{\eta^2}{4}R^2) $.}
\end{equation} 
Once \eqref{eq-barrier-phi} is proved, then property (iii) follows from the simple calculation that $R^2\sL\left[(a_L+\Lambda+1)\frac{t}{R^2}\right]=-(a_L+\Lambda+1)$ in $K_{\beta_1R, \beta_2R^2}(z_o, \beta_2R^2)$.
Now we use the identity 
$$\sL[\vp(u(x),t)]=\partial_u\vp(u,t)\sL u+\partial_{uu}\vp(u,t)\langle A_{x,t}\D u,\D u\rangle -\partial_t\vp(u,t)$$ 
to obtain 
\begin{align*}
\sL \phi &=\frac{2d}{R^2}\p_sh\left(\frac{d^2}{R^2},\frac{t}{R^2}\right)Ld\\
&+\left\{\frac{2}{R^2}\p_sh +\frac{4d^2}{R^4}\p_{ss}h\right\}\left(\frac{d^2}{R^2},\frac{t}{R^2}\right)\langle A_{x,t}\D d,\D d\rangle-\frac{1}{R^2}\p_th\left(\frac{d^2}{R^2},\frac{t}{R^2}\right).
\end{align*}
Since $d \cdot Ld\leq a_L$ and $\lambda\leq \langle A_{x,t}\D d,\D d\rangle \leq \Lambda$ in $M\bs \Cut (z_o), $ we have   that 
 \begin{align*}
&\frac{(\beta_1^2R^2-d^2)^2}{(-\phi)}\sL \phi \\&= (\beta_1^2R^2-d^2)\left\{2l+(\beta_1^2R^2-d^2)\frac{2\alpha}{t}\right\}d\sL d\\ 
&-\left\{l(l-1)4d^2 +2l(\beta_1^2R^2-d^2)\frac{4\alpha d^2}{t}+(\beta_1^2R^2-d^2)^2\frac{4\al^2d^2}{t^2}\right\}\langle A_{x,t}\D d,\D d\rangle\\
&+(\beta_1^2R^2-d^2)\left\{ 2l +(\beta_1^2R^2-d^2)\frac{2\al}{t}\right\}\langle A_{x,t}\D d,\D d\rangle\\
&+(\beta_1^2R^2-d^2)^2\frac{\al d^2}{t^2}-(\beta_1^2R^2-d^2)^2\left(\frac{n}{2t}+\frac{m}{R^2}\right)\\
&\leq  2l(\beta_1^2R^2-d^2)(a_L+\Lambda)+(\beta_1^2R^2-d^2)^2\left\{\frac{2\al}{t}(a_L+\Lambda)+\frac{\al d^2}{t^2}\right\}\\
&-l(l-1)4d^2\lambda-(\beta_1^2R^2-d^2)^2\left(\frac{4\al^2d^2}{t^2}\lambda+\frac{n}{2t}\right)\\&-2l(\beta_1^2R^2-d^2)\frac{4\al d^2}{t}\lambda-\frac{m}{R^2}(\beta_1^2R^2-d^2)^2\quad\mbox{a.e. in   $K_{\beta_1R,\,\beta_2R^2}(z_o, \beta_2R^2)$}.
\end{align*} By choosing
\begin{equation}\label{eq-choice-const}
\begin{split}
\al&:=
\frac{1}{4\lambda},\,\,\,\,\frac{2\beta_1^2}{\eta^2\lambda} (a_L+\Lambda)+1\leq l:= 2l'+1 \,\,\,\mbox{(for some  $l'\in\N$)} ,\\
m&:=2\cdot\max\left\{\frac{8\al }{\eta^2}(a_L+\Lambda),\,\, \frac{2l (a_L+\Lambda)}{\beta_1^2-\frac{\eta^2}{4}}\right\},
\end{split}\end{equation} 
we deduce
$$\frac{(\beta_1^2R^2-d^2)^2}{(-\phi)}\sL \phi\leq 0\quad\mbox{a.e. in $K_{\beta_1R, \beta_2R^2}(z_o, \beta_2R^2) \bs K_{\frac{\eta}{2} R}(z_o, \frac{\eta^2}{4}R^2)$.}$$ Indeed,  we 
divide the domain $K_{\beta_1R, \beta_2R^2}(z_o, \beta_2R^2) \bs K_{\frac{\eta}{2} R}(z_o, \frac{\eta^2}{4}R^2)$ into three regions such that
$$K_{\beta_1R, \beta_2R^2}(z_o, \beta_2R^2) \bs K_{\frac{\eta}{2} R}(z_o, \frac{\eta^2}{4}R^2)=:A_1\cup A_2\cup A_3, $$ where
$A_1:= \{0\leq \frac{t}{R^2}\leq \frac{\eta^2}{4},  \frac{\eta}{2}\leq \frac{d}{R}\leq\beta_1 \}, $ $A_2:=\{\frac{\eta^2}{4}\leq \frac{t}{R^2}\leq \beta_2,  \frac{\eta}{2}\leq \frac{d}{R} \leq  \beta_1 \}$ and  $A_3:=\{\frac{\eta^2}{4}\leq \frac{t}{R^2}\leq \beta_2, 0\leq \frac{d}{R}\leq\frac{\eta}{2} \}$.     We can  check that  $$\frac{(\beta_1^2R^2-d^2)^2}{(-\phi)}\sL \phi\leq 0\quad\mbox{a.e. in  $K_{\beta_1R, \beta_2R^2}(z_o, \beta_2R^2) \bs K_{\frac{\eta}{2} R}(z_o, \frac{\eta^2}{4}R^2)$}$$  by choosing $\al$ and $l$ large  in $A_1$,  $m$ large   in $A_2$ and  $A_3$ as in  \eqref{eq-choice-const}. Therefore, we have proved \eqref{eq-barrier-phi}.

 From the assumption on $\psi$,  we have  that for a.e. $(x,t)\in K_{\beta_1R,\,\beta_2R^2}(z_o, \beta_2R^2)$, 
\begin{align*}
R^2\sL v(x,t)
&=2\p_s\psi\left(\frac{d^2}{R^2},\frac{t}{R^2}\right)d\cdot Ld +\left\{2\p_s\psi+\frac{4d^2}{R^2}\p_{ss}\psi\right\}\left(\frac{d^2}{R^2},\frac{t}{R^2}\right)\langle A_{x,t}\D d,\D d\rangle\\&- \p_t\psi\left(\frac{d^2}{R^2},\frac{t}{R^2}\right)\\
&\leq  \displaystyle\sup_{[0, \beta_1^2]\times[0,\beta_2]}\left\{2a_L{\p_s\psi}+\Lambda\left(2\p_s\psi+4s|\p_{ss}\psi|\right)+|\p_t\psi| \right\}(s,t)< C_\eta.
\end{align*} This proves property (iv).

In order to show (ii), we take  $A>0$ large enough  so that for $(x,t)\in K_{2R}(z_o, \beta_2R^2)$, 
\begin{align*}
v(x,t)&\leq -Ae^{-\beta_2m}\left(1-\frac{4}{\beta_1^2}\right)^l\frac{1}{{(4\pi\beta_2)}^{n/2}}e^{-4\al/\eta^2}+(a_L+\Lambda+1)\beta_2\leq 0.
\end{align*}
This finishes the proof of the lemma. 
\end{proof}

Now we  apply  Lemma \ref{lem-abp-type} to $u+v_\eta$ with $v_\eta$ constructed in Lemma \ref{lem-barrier} and translated in time.   
Since the barrier function $v_\eta(x,t)=\psi_\eta\left(\frac{d_{z_o}^2(x)}{R^2},\frac{t}{R^2}\right)$ is not smooth on  $\Cut(z_o)$,  
we need to approximate $v_\eta$ by a sequence of  smooth functions as Cabr\'{e}'s approach at \cite{Ca}. We recall that the cut locus of $z_o$ is  closed and has measure zero.  It is not hard to verify the following lemma and   we just  refer to \cite{Ca} Lemmas 5.3, 5.4. 

\begin{lemma}\label{lem-barrier-approx}
 Let $z_o\in M,\, R>0$ and let $\psi:\R^+\times[0,T]\to\R$ be a smooth  function such that 
$\psi(s,t)$ is nondecreasing with respect to $s$ for any $t\in[0,T]$. Let $v(x,t):=\psi\left({d^2_{z_o}(x)} , t\right)$. Then there exist a smooth function $0\leq\zeta(x)\leq1$ on $M$ satisfying  $$\zeta\equiv1 \,\,\,\,\mbox{in $B_{\beta_1R}(z_o)$ and }\,\,\,\,\supp\zeta\subset B_{\frac{10}{\eta}R}(z_o) $$ and a sequence $\{w_k\}_{k=1}^\infty$ of smooth functions in $M\times[0,T] $  such that
\begin{equation*} 
\left\{
\begin{array}{ll}
 w_k\to \zeta v\qquad &\mbox{uniformly in $M\times[0,T]$, }\\ 
 \p_tw_k\to \zeta \p_tv\qquad &\mbox{uniformly in $M\times[0,T]$, }\\
 D^2_x w_k\leq C \mathrm{Id}\qquad &\mbox{  in $M\times[0,T]$, }\\
 D^2_xw_k\to D_x^2v\qquad &\mbox{a.e. in $B_{\beta_1R}(z_o)\times[0,T]$, }
 \end{array}\right. \qquad \qquad
\end{equation*}
where the constant $C>0$ is independent of $k$.
\end{lemma}

\begin{lemma} \label{lem-decay-est-1-step}Suppose that $M$ satisfies the conditions \eqref{cond-M-1},\eqref{cond-M-2}. Let $z_o\in M, R>0$,  and $0<\eta<1$.
Let $u$ be a smooth function such that $\sL u \leq f$  in $K_{\alpha_1R,\, \alpha_2R^2}(z_o,4R^2)$ such that
$$u\geq 0\quad\mbox{in}\quad K_{\alpha_1R,\,\alpha_2R^2}(z_o, 4R^2)\bs  K_{\beta_1R,\,\beta_2R^2}(z_o,4R^2) $$
and $$\inf_{K_{2R}(z_o,4R^2)}u\leq1. $$
Then, there exist uniform constants $M_\eta>1, 0<\mu_\eta<1$,  and $0<\ve_\eta<1$ such that 
\begin{equation}\label{eq-decay-est-1-step}
\frac{\left|\left\{u\leq M_{\eta}\right\}\cap K_{\eta R}(z_o ,0)\right|}{\left|K_{\alpha_1R,\, \alpha_2R^2}(z_o, 4R^2)\right|}\geq \mu_\eta,\end{equation}
provided 
\begin{equation}\label{cond-f}
R^2\left(\fint_{K_{\alpha_1R,\,\alpha_2R^2}(z_o, 4R^2)} |f^+|^{n+1}\right)^{\frac{1}{n+1}}\leq \ve_\eta,
\end{equation}
where $ M_\eta>0,\, 0<\ve_\eta, \,\mu_\eta<1 $ depend only on $\eta, n, \lambda,\Lambda$ and $a_L$.
\end{lemma}
\begin{proof}
Let $v_\eta$ be the barrier function in Lemma \ref{lem-barrier} after translation in time (by $-\eta^2R^2$) and let $\{w_k\}_{k=1}^\infty$ be a sequence of smooth functions approximating $v_\eta$ as in  Lemma \ref{lem-barrier-approx}. We notice that $u+v_\eta\geq0 $ in $K_{\alpha_1R, \,\alpha_2R^2}(z_o, 4R^2)\bs K_{\beta_1R, \,\beta_2R^2}(z_o, 4R^2)$ and $\displaystyle\inf_{K_{2R}(z_o,4R^2)} (u+v_{\eta})\leq1$. 
Thanks to the uniform convergence of $w_k$ to $\zeta v_\eta$, we consider a sequence $\{\ve_k\}_{k=1}^{\infty}$ converging to $0$ such that $\displaystyle\sup_{K_{2R}(z_o, 4R^2)}w_k\leq \ve_k$ and  $$w_k \geq -\ve_k\,\,\,\mbox{in}\,\,\, {K_{\alpha_1R,\,\alpha_2R^2}(z_o,4R^2)\bs K_{\beta_1R,\,\beta_2R^2}(z_o,4R^2)}, $$ and 
define 
$$\overline{w}_k:= \frac{u+w_k+\ve_k}{1+2\ve_k}. $$
Then $\overline{w}_k$ satisfies the hypotheses of Lemma \ref{lem-abp-type} (after translation in time by $4R^2$). 
Now we replace $u $ by $\overline w_k$ in \eqref{eq-abp-type} and then the uniform convergence implies that for a given $0<\delta<1$,  we have  
\begin{align*}
|B_R(z_o)|R^2
&\leq   C(\eta,n,\lambda)\int_{\{u+v_\eta\leq M_\eta+\delta\}\cap K_{\beta_1R, \,\beta_2R^2}(z_o,4R^2)}\left\{\left(R^2\sL \overline w_k + a_{L}+\Lambda+1 \right)^+ \right\}^{n+1} 
\end{align*}
if $k$ is sufficiently large.  Since $D_x^2w_k\leq C{\rm Id}$ and $|\p_tw_k|<C$ uniformly in $k$ on $K_{\beta_1 R,\beta_2 R^2}(z_o,4R^2)$, we  use the dominated convergence theorem to let $k$ go to $+\infty$. Letting $\delta$ go to $0$,  we  obtain 
\begin{align*}
|B_R(z_o)|\cdot R^2&\leq C(\eta,n,\lambda)\int_{\{u+v_\eta\leq M_\eta\}\cap K_{\beta_1R, \beta_2R^2}(z_o, 4R^2)}\left\{\left(R^2\sL[u+v_\eta]+ a_{L}+\Lambda+1 \right)^+ \right\}^{n+1}\\
&= C(\eta,n,\lambda)\int_{E_1\cup E_2} \left\{\left(R^2\sL[u+v_\eta]+ a_{L}+\Lambda+1 \right)^+ \right\}^{n+1},
\end{align*}
where  $E_1:=\{u+v_\eta\leq M_\eta\}\cap \left(K_{\beta_1R,\,\beta_2R^2}(z_o,4R^2)\bs K_{\eta R}(z_o,0)\right) $ and $E_2:=\{u+v_\eta\leq M_\eta\}\cap K_{\eta R}(z_o,0)$. 
From properties (iii) and (iv) of $v_\eta$ in Lemma \ref{lem-barrier} and Bishop's volume comparison theorem in Lemma \ref{lem-jac-bishop}, we deduce that
\begin{align*}{|K_{\alpha_1R,\,\alpha_2R^2}(z_o,4R^2)|^{\frac{1}{n+1}}} 
&\leq C_\eta \left\|\left(R^2\sL u\right)^+\right\|_{L^{n+1}\left(K_{\beta_1R, \beta_2R^2}(z_o,4R^2)\right)}+C_\eta\left\|\chi_{E_{2}}\right\|_{L^{n+1}\left(K_{\beta_1R, \,\beta_2R^2}(z_o,4R^2)\right)} \\&\leq C_\eta ||R^2f^+||_{L^{n+1}\left(K_{\alpha_1R,\,\alpha_2R^2}(z_o,4R^2)\right)} +C_\eta\left|\left\{u+v_\eta\leq M_\eta\right\}\cap K_{\eta R }(z_o,0)\right|^{\frac{1}{n+1}},\end{align*}
where $C_\eta>0$ depends only on $n,\lambda$ and $\eta>0$. 
We note that $\{u\leq M_\eta-v_\eta\}\subset \{u\leq M_\eta+C_\eta\}$ from (v) in Lemma \ref{lem-barrier}.
 Therefore, by taking $$  \ve_\eta=\frac{1}{2C_\eta},\,\,\, M'_\eta= M_\eta+C_\eta\,\,\,\,\mbox{and}\,\,\,\,\mu_\eta^{\frac{1}{n+1}}=\frac{1}{2C_\eta}, $$  we conclude that  $\displaystyle\frac{\left|\left\{u\leq M'_{\eta}\right\}\cap K_{\eta R}(z_o,0) \right|}{\left|K_{\alpha_1R,\,\alpha_2R^2}(z_o,4R^2)\right|}\geq \mu_\eta>0.$\end{proof}


Using iteration of Lemma \ref{lem-decay-est-1-step}, we have the following corollaries. 

\begin{cor}\label{cor-decay-1}
Suppose that $M$ satisfies the conditions \eqref{cond-M-1},\eqref{cond-M-2}. Let $z_o\in M$ and $0<\eta<1$.  
  For $i\in \N$,  let $\overline R_i:=\left(\frac{2}{\eta}\right)^{i-1} R$ and  $\overline t_i:=\sum_{j=1}^i 4\overline R_j^2$. 
Let $u$ be a nonnegative  smooth function such that $\sL u\leq f$ in $\bigcup_{i=1}^k K_{\alpha_1\overline R_i, \alpha_2\overline R_i^2}(z_o,\overline t_i)$ for some $k\in\N$.   
 We assume that for $h>0$,  $\displaystyle\inf_{\bigcup_{i=1}^kK_{2\overline R_i}(z_o,\overline t_i)}u\leq h$ and $$
 \overline R_i^2\left(\fint_{ K_{\alpha_1\overline R_i,  \alpha_2\overline R_i^2}(z_o, \overline t_i)}|f^+|^{n+1}\right)^{\frac{1}{n+1}}\leq \ve_\eta h M_\eta^{k-i},\,\,\,\forall 1\leq i\leq k.
$$ 
 Then we have 
\begin{equation}\label{eq-cor-decay-1}
\frac{\left|\left\{u\leq h M^k_{\eta}\right\}\cap K_{\eta R}(z_o,0)\right|}{\left|K_{\alpha_1R,\,\alpha_2R^2}(z_o,4R^2)\right|}\geq \mu_\eta,\end{equation}
where $M_\eta, \ve_\eta,\mu_\eta$ are the same uniform constants as  in Lemma \ref{lem-decay-est-1-step}.
\end{cor}
\begin{proof} 
We may assume $h=1$ since $v:=\frac{u}{h}$ satisfies $\sL v=\frac{1}{h}\sL u\leq \frac{f}{h}$. We use the induction on $k$ to show the lemma.  When $k=1$,  it is immediate from Lemma \ref{lem-decay-est-1-step}.  

Now suppose that \eqref{eq-cor-decay-1} is true for $k-1$. By assumption, we find a $j_o\in\N$ such that $1\leq j_o\leq k$ and 
$\displaystyle\inf_{ K_{2\overline R_{j_o}}(z_o, \overline t_{j_o})}u =\inf_{\bigcup_{i=1}^kK_{2\overline R_{i}}(z_o, \overline t_{i})}u\leq 1$. Define $v:={u}/{M^{k-j_o}_{\eta}}$.   Then $v$ satisfies  that $\displaystyle\sL v\leq {f}/{M^{k-j_o}_{\eta}},\, \inf_{K_{2\overline R_{j_o}}(z_o, \overline t_{j_o})}v\leq 1$   and 
$$
\overline R_{j_o}^2\left(\fint_{ K_{\alpha_1\overline R_{j_o}, \,\alpha_2\overline R_{j_o}^2}(z_o,\overline t_{j_o})}\left|f^+/ M^{k-j_o}_\eta\right|^{n+1}\right)^{\frac{1}{n+1}}
\leq \ve_\eta.
$$
 Applying Lemma \ref{lem-decay-est-1-step} to $v$ in $K_{\alpha_1\overline R_{j_o},\,\alpha_2\overline R_{j_o}^2}(z_o,\overline t_{j_o})$,  we deduce 
$$\frac{\left|\left\{v\leq M_{\eta}\right\}\cap K_{\eta\overline R_{j_o}}(z_o,\overline t_{j_o}-4R_{j_o}^2)\right|}{\left| K_{\alpha_1\overline R_{j_o}, \,\alpha_2\overline R_{j_o}^2}(z_o,\overline t_{j_o})\right|}=\frac{\left|\left\{v\leq M_{\eta}\right\}\cap K_{2\overline R_{j_o-1}}(z_o,\overline t_{j_o-1})\right|}{\left| K_{\alpha_1\overline R_{j_o}, \,\alpha_2\overline R_{j_o}^2}(z_o,\overline t_{j_o})\right|}\geq \mu_\eta>0$$ which implies that $\displaystyle\inf_{\bigcup_{i=1}^{j_o-1}K_{2\overline R_i}(z_o,\overline t_i)}u \leq \inf_{K_{2\overline R_{j_o-1}}(z_o,\overline t_{j_o-1})} u\leq M^{k-j_o+1}_{\eta}$. Therefore, we use the induction hypothesis for $j_o-1 (\leq k-1)$ to conclude 
$$\frac{\left|\left\{u/M_\eta^{k-j_o+1}\leq M^{j_o-1}_{\eta}\right\}\cap K_{\eta R}(z_o,0)\right|}{\left|K_{\alpha_1R,\,\alpha_2R^2}(z_o , 4R^2)\right|}\geq \mu_\eta>0,$$
which implies \eqref{eq-cor-decay-1}.
\end{proof}

We remark that Lemma \ref{lem-decay-est-1-step} and  Corollary~\ref{cor-decay-1} hold  for any $M'_\eta\geq M_\eta$.  
The following is a simple technical lemma that will be used in the proof of Proposition \ref{cor-decay-2}.
 
\begin{lemma}\label{lem-ext-time}
Let $A,D>0$ and  $\ve>0$. Let $u$  be a nonnegative smooth function such that $\sL u \leq f$ in $ \overline{B_R(z_o)}\times(-AR^2,0]$  with 
$$R^2\left(\fint_{B_R(z_o)\times(-AR^2,0]}|f^+|^{n+1}\right)^{\frac{1}{n+1}}\leq\ve. $$
 Then, there exists a sequence $u_k$ of nonnegative smooth functions in $B_R(z_o)\times(-AR^2,DR^2]$ such that $u_k$ converges to $u$ locally uniformly in $B_R(z_o)\times(-AR^2,0]$ and $\sL u_k \leq g_k$ in $B_R(z_o)\times(-AR^2,DR^2]$ with 
 $$R^2\left(\fint_{B_R(z_o)\times(-AR^2,DR^2]}|g_k^+|^{n+1}\right)^{\frac{1}{n+1}}\leq\ve. $$
\end{lemma}
\begin{proof}
First, 
we 
define  for $(x,t)\in B_R(z_o)\times(-\infty, DR^2]$, 
\begin{equation*} 
\overline u(x,t):= \left\{
\begin{array}{ll}
0\qquad &\mbox{for $\,\,t\in  (-\infty,-AR^2]$, }\\ 
u(x,t)\qquad &\mbox{for $\,\,t\in  (-AR^2,0]$, }\\ 
 u(x,0)+St\qquad &\mbox{for $\,\,t\in(0, DR^2]$},
 \end{array}\right. \qquad \qquad
\end{equation*}
where $\displaystyle S:=\sup_{B_R(z_o)}\left\{(\sL u)^+(x,0)+|u_t(x,0)|\right\}$. Then $\overline u$ is Lipschitz continuous with respect to time  in $B_R(z_o)\times(-AR^2,DR^2]$ and satisfies 
\begin{equation*} 
\sL \overline u (x,t)\leq \overline f(x,t):=\left\{
\begin{array}{ll}
0\quad &\mbox{for $ t\in  (-\infty,-AR^2)$, }\\ 
f(x,t)\quad &\mbox{for $ t\in  (-AR^2,0)$, }\\ 
  \sL u(x,0)+u_t(x,0)-S\leq0\,\,\,&\mbox{for $t\in(0,DR^2]$}.
 \end{array}\right. \qquad \qquad
\end{equation*}
Let $\ve_k>0$ converge to $0$ as 
$k\to+\infty,$   and let $\vp$ be a nonnegative smooth function such that $\vp(t)=0$ for $t\not\in(0,1)$ and $\int_\R\vp(t)dt=1$.   We define  $\displaystyle\vp_k(t):=\frac{1}{\ve_k}\vp\left(\frac{t}{\ve_k}\right)$ and  
$$u_k(x,t):=  \int_\R\overline u(x,s)\vp_k(t-s)ds,\,\,\, \forall(x,t)\in B_R(z_o)\times(-\infty, DR^2], $$ where we notice that the above integral is calculated over $(t-\ve_k, t)\subset \R$. Then, a smooth function $u_k$  satisfies  $$\sL u_k (x,t)=\int_\R\sL \overline u(x,s)\vp_k(t-s)ds  \leq g_k(x,t),\quad\mbox{$\forall(x,t)\in B_R(z_o)\times(-\infty,DR^2],$ }$$ where $g_k(x,t):=\displaystyle\int_\R \overline f^+(x,s)\vp_k(t-s)ds  \geq 0$. We also have 
\begin{align*}
R^2\left(\fint_{B_R(z_o)\times(-AR^2,DR^2]}|g_k^+|^{n+1}\right)^{\frac{1}{n+1}}
&\leq \frac{R^2}{\left\{|B_R(z_o)|\cdot(A+D)R^2\right\}^{\frac{1}{n+1}}}||\overline f^+||_{L^{n+1}\left(B_R(z_o)\times(-AR^2-\ve_k,DR^2-\ve_k]\right)}\\
&\leq\frac{R^2}{\left\{|B_R(z_o)|\cdot(A+D)R^2\right\}^{\frac{1}{n+1}}}||\overline f^+||_{L^{n+1}\left(B_R(z_o)\times(-AR^2,0]\right)}\\
&\leq \left(\frac{A}{A+D}\right)^{\frac{1}{n+1}}\ve  < \ve,
\end{align*}which  finishes the proof.
\end{proof}

\begin{prop}\label{cor-decay-2} 
Suppose that $M$ satisfies the conditions \eqref{cond-M-1},\eqref{cond-M-2}. Let $z_o\in M, R>0, 0<\eta<\frac{1}{2}$ and $\tau\in[3,16]. $   
Let $u$ be a nonnegative smooth function such that $\sL u \leq f $  in $
 B_{\frac{49}{\eta^3}R}(z_o)\times\left(-3R^2,\frac{\tau R^2}{\eta^2}\right]$. 
Assume that  
$\displaystyle\inf_{B_{R}(z_o)\times\left[\frac{2R^2}{\eta^2}, \frac{\tau R^2}{\eta^2} \right]}u\leq 1$ and 
$$ R^2\left( \fint_{B_{\frac{49}{\eta^3} R}(z_o)\times\left(-3R^2,\frac{\tau R^2}{\eta^2}\right]}|f^+|^{n+1}\right)^{\frac{1}{n+1}} \leq  \ve'_\eta$$
for a uniform constant $0<\ve'_\eta<1.$ 
Let $r>0$ satisfy $\left(\frac{\eta}{2}\right)^{N}R\leq r <\left(\frac{\eta}{2}\right)^{N-1}R$ for some $N\in\N$  and let $(z_1,t_1)$ be a point such that $d(z_o,z_1)<R$ and $|t_1|< R^2$. 
  Then there exists a uniform constant $ M'_\eta>1$ (independent of $r, N, z_1$ and $t_1$)  such that 
  \begin{equation*}
\frac{\left|\left\{u\leq  {M'_{\eta}}^{N+2} \right\}\cap K_{\eta r}(z_1,t_1)\right|}{\left|K_{\alpha_1r,\,\alpha_2r^2}(z_1 , t_1+4r^2)\right|}\geq \mu_\eta>0,\end{equation*}
where $0<\mu_\eta<1$  is the constant in Lemma \ref{lem-decay-est-1-step}.
\end{prop}
\begin{proof} (i)  From Lemma \ref{lem-ext-time}, we approximate $u$ by nonnegative smooth functions $u_k,$ which  are  defined on $ B_{\frac{48}{\eta^3} R}(z_o)\times\left(-3R^2,\frac{64R^2}{(4-\eta^2)\eta^6}+R^2\right]. $ We  find  functions $u_k$ and $g_k$ such that $u_k$ converges locally uniformly to $u$ in $B_{\frac{48}{\eta^3} R}(z_o)\times\left(-3R^2, \frac{\tau R^2}{\eta^2}\right],$  and satisfies $$\sL u_k\leq g_k\quad\mbox{in}\,\,\,B_{\frac{48}{\eta^3} R}(z_o)\times\left(-3R^2,\frac{64R^2}{(4-\eta^2)\eta^6}+R^2\right], $$   and 
$$ R^2\left( \fint_{B_{\frac{48}{\eta^3} R}(z_o)\times\left(-3R^2,\frac{64R^2}{(4-\eta^2)\eta^6}+R^2\right]} |g_k^+|^{n+1}\right)^{\frac{1}{n+1}} \leq \frac{49}{48}\ve'_\eta< 2\ve'_\eta $$ by using the volume comparison theorem and Lemma \ref{lem-ext-time}.  
For a small $\delta>0$,  we consider  $w_k:=\displaystyle\frac{u_k}{1+\delta}$ and then   for large $k,$ $w_k$ satisfies  $\displaystyle\inf_{B_{R}(z_o)\times\left[\frac{2R^2}{\eta^2}, \frac{\tau R^2}{\eta^2} \right]}w_k\leq 1,$  $\sL w_k\leq g_k\,\,\,\mbox{in}\,\,\,B_{\frac{48}{\eta^3} R}(z_o)\times\left(-3R^2,\frac{64R^2}{(4-\eta^2)\eta^6}+R^2\right], $  and 
$$ R^2\left( \fint_{B_{\frac{48}{\eta^3} R}(z_o)\times\left(-3R^2,\frac{64R^2}{(4-\eta^2)\eta^6}+R^2\right]} |g_k^+|^{n+1}\right)^{\frac{1}{n+1}} < 2\ve'_\eta, $$
 according to the local uniform convergence of $u_k$ to $u$ in Lemma \ref{lem-ext-time}.  So if we show the proposition  for $w_k$,  the local uniform convergence will imply that the result holds for $u$  by letting $k\to+\infty$ and $\delta\to0$. 
 Now we assume that $u$ is a nonnegative smooth function in $ B_{\frac{48}{\eta^3} R}(z_o)\times\left(-3R^2,\frac{64R^2}{(4-\eta^2)\eta^6}+R^2\right]$ satisfying the same hypotheses  as $w_k.$ 

(ii) We use  Corollary \ref{cor-decay-1} so we need to check the two  hypotheses with $k=N+2$ and $h=1.$ As in the corollary,   we define  for $i\in \N$,   $$\overline r_i:=\left(\frac{2}{\eta}\right)^{i-1} r\quad \mbox{and}\quad \overline t_i:=t_1+\sum_{j=1}^i 4\overline r_j^2. $$
Using the conditions on $r, $ $z_1$,  and $t_1,$  simple computation says that  for $0<\eta<1/2,$
\begin{align*}&B_{2\overline r_{N+1}}(z_1)\supset B_{2R}(z_1)\supset B_{R}(z_o), 
\\&\overline t_N  <R^2+\frac{16R^2}{4-\eta^2}<\frac{2R^2}{\eta^2} <\frac{16R^2}{\eta^2}< -R^2+ \frac{4(4+\eta^2)R^2}{\eta^2}  < \overline t_{N+2}.\end{align*}  Thus 
  we have $ B_{2\overline r_{N+1}}(z_1)\times(\overline t_N,\overline t_{N+2})\supset B_{R}(z_o)\times\left[\frac{2R^2}{\eta^2}, \frac{16R^2}{\eta^2} \right]\supset B_{R}(z_o)\times\left[\frac{2R^2}{\eta^2}, \frac{\tau R^2}{\eta^2} \right]$ for $0<\eta<\frac{1}{2}$  
and hence  $\displaystyle\inf_{\bigcup_{i=1}^{N+2}K_{2\overline r_i}(z_1,\overline t_i)}u\leq \inf_{\bigcup_{i=N+1}^{N+2}K_{2\overline r_i}(z_1,\overline t_i)}u\leq1$. We remark that $\overline r_{N+2} $ is comparable to $R$. 

Now, it suffices to show 
 for some large $ M'_\eta\geq M_\eta$, and small $0<\ve'_\eta<\ve_\eta,$  we have
\begin{equation}\label{eq-cor-decay-2-f}
\overline r_i^2\left(\fint_{K_{\alpha_1\overline r_i , \alpha_2\overline r_i^2}(z_1,\overline t_i)}|f^+|^{n+1}\right)^{\frac{1}{n+1}}\leq \ve_\eta  { M'_{\eta}}^{N+2-i},\quad\forall 1\leq i\leq N+2,\end{equation} where $M_\eta$ and $\ve_\eta$ are the constants in Corollary \ref{cor-decay-1}. 
 We notice that  $
 B_{\beta_1\overline r_{N+2}}(z_o)\subset
  B_{\alpha_1\overline r_{N+2}}(z_1)\subset B_{\frac{12}{\eta}\cdot\frac{4}{\eta^2}R}(z_o) $ and $$\bigcup_{i=1}^{N+2}K_{\alpha_1\overline r_i , \alpha_2\overline r_i^2}(z_1,\overline t_i)\subset B_{\frac{48}{\eta^3} R}(z_o)\times\left(-3R^2,\frac{64R^2}{(4-\eta^2)\eta^6}+R^2\right]$$  since $d(z_o,z_1)<R,  |t_1|< R^2$ and $\frac{2}{\eta}R\leq  \overline r_{N+2} <\frac{4}{\eta^2}R$. 
 Then for $i=1,2,\cdots, N+2$,  we have
\begin{align*}
\overline r_i^{2(n+1)}\fint_{K_{\alpha_1\overline r_i , \alpha_2\overline r_i^2}(z_1,\overline t_i)}|f^+|^{n+1}
&
\leq  \left(\frac{4}{\eta^2}\right)^{2(n+1)}\frac{R^{2(n+1)}}{|K_{\alpha_1\overline r_i, \alpha_2\overline r_i^2}(z_1,\overline t_i)|}||f^+||^{n+1}_{L^{n+1}\left(B_{\frac{48}{\eta^3} R}(z_o)\times\left(-3R^2,\frac{64R^2}{(4-\eta^2)\eta^6}+R^2\right]\right)} 
\\&\leq \left(\frac{4}{\eta^2}\right)^{2(n+1)}(2\ve'_\eta)^{n+1} \frac{\left|B_{\frac{48}{\eta^3} R}(z_o)\times\left(-3R^2,\frac{64R^2}{(4-\eta^2)\eta^6}+R^2\right]\right|}{|K_{\alpha_1\overline r_i , \alpha_2\overline r_i^2}(z_1,\overline t_i)|}
\\&\leq C(n,\eta) {\ve'_\eta}^{n+1} \frac{| B_{\frac{48}{\eta^3}  R}(z_o)| R^2}{|B_{\alpha_1\overline r_{i}}(z_1)| \overline r_{i}^2}
\leq C(n,\eta) {\ve'_\eta}^{n+1} \frac{\left| B_{\beta_1\overline r_{N+2}}(z_o)\right| {\overline r_{N+2}^2}}{|B_{\alpha_1\overline r_{i}}(z_1)| \overline r_{i}^2},
\end{align*} 
where we use that $\frac{2}{\eta}R\leq \overline r_{N+2} <\frac{4}{\eta^2}R$ and  the volume comparison theorem in the last inequality and the constant $C(n,\eta)>0$ depending only on $n$ and $\eta,$ may change from line to line. Since $d(z_o,z_1)<R$,  we use  the volume comparison theorem again to obtain
\begin{align*}
\overline r_i^{2(n+1)}\fint_{K_{\alpha_1\overline r_i , \alpha_2\overline r_i^2}(z_1,\overline t_i)}|f^+|^{n+1}
& \leq C(n,\eta) {\ve'_\eta}^{n+1} \frac{\left| B_{\beta_1\overline r_{N+2}}(z_o)\right| {\overline r_{N+2}^2}}{|B_{\alpha_1\overline r_{i}}(z_1)| \overline r_{i}^2}\\&\leq C(n,\eta) {\ve'_\eta}^{n+1} \frac{\left| B_{\alpha_1\overline r_{N+2}}(z_1)\right|\overline r_{N+2}^2}{|B_{\alpha_1\overline r_{i}}(z_1)| \overline r_{i}^2}\\&\leq C(n,\eta){\ve'_\eta}^{n+1}\left(\frac{\overline r_{N+2}}{\overline r_i}\right)^{n+2}\leq C(n,\eta){\ve'_\eta}^{n+1}\left(\frac{2}{\eta}\right)^{(n+2)(N+2-i)}.
\end{align*}
We select $ M'_\eta>M_\eta$ large  and $0< \ve'_\eta<\ve_\eta$ small enough to satisfy 
$$C(n,\eta){\ve'_\eta}^{n+1} \left(\frac{2}{\eta}\right)^{(n+2)(N+2-i)}\leq \ve_\eta^{n+1} { M'_{\eta}}^{(n+1)(N+2-i)},\,\,\,\forall 1\leq i\leq N+2, $$ which proves \eqref{eq-cor-decay-2-f}. 
Therefore, Corollary \ref{cor-decay-1} (after translation in time by $t_1$) gives   \begin{equation*}
\frac{\left|\left\{u\leq {M'_{\eta}}^{N+2} \right\}\cap K_{\eta r}(z_1,t_1)\right|}{\left|K_{\alpha_1r,\,\alpha_2r^2}(z_1 , t_1+4r^2)] \right|}\geq \mu_\eta>0.\end{equation*}
\end{proof}

\section{Parabolic version of the Calder\'{o}n-Zygmund decomposition}\label{sec-cz}
Throughout this section, we assume that a complete Riemannian manifold $M$ satisfies the condition \eqref{cond-M-1}. 
We introduce a parabolic version of the Calder\'{o}n-Zygmund lemma ( Lemma \ref{lem-cz} )  to prove power decay of super-level sets in Lemma \ref{lem-decay-est} (see   \cite{W, Ca, CC}).
Christ \cite{Ch} proved that the following theorem holds for  so-called "spaces of homogeneous type",  which is a generalization of Euclidean dyadic decomposition. In harmonic analysis, a metric space  $X$ is called  a space of homogeneous type when $X$ equips a nonnegative Borel measure $\nu$ satisfying the doubling property
$$\nu(B_{2R}(x))\leq  A_1 \nu(B_R(x)) <+\infty,\,\,\,\,\forall x\in X,\,\, R>0, $$
for some constant $A_1$ independent of $x$ and $R$.  From Bishop's volume comparison (Lemma \ref{lem-jac-bishop}), a complete Riemannian manifold  $M$ satisfying the condition \eqref{cond-M-1} is a space of homogeneous type with $A_1=2^n$.

\begin{thm}[Christ]\label{thm-christ}
There exist a countable collection $\{Q^{k,\alpha}\subset M : k\in\Z,\alpha\in I_k\}$ of open subsets of $M$ and positive constants $0<\delta_0<1$, $c_1$ and $c_2$ (with $2c_1\leq c_2$ ) that depend only on $n$, such that
\begin{enumerate}[(i)]
\item $\left| M \backslash \bigcup_{\alpha}Q^{k,\alpha}\right|=0$ for $k\in\Z$,
\item if $l\leq k$, $\alpha\in I_k$, and $\beta\in I_l$, then either $Q^{k,\alpha}\subset Q^{l,\beta}$ or $Q^{k,\alpha}\cap Q^{l,\beta}=\emptyset$,
\item for any $(k,\alpha)$ and any $l<k$, there is a unique $\beta$ such that $Q^{k,\alpha}\subset Q^{l,\beta}$,
\item $\text{diam}(Q^{k,\alpha})\leq c_2\delta^k_0 $,
\item any $Q^{k,\alpha}$ contains some ball $B_{c_1\delta^k_0}(z^{k,\alpha})$.
\end{enumerate}
\end{thm}
For convenience, we will use the following notation. 
\begin{definition}[Dyadic cubes on  $M$]
\label{def-d-cube-m}\item
 \begin{enumerate}[(i)]
\item  The open set $ Q=Q^{k,\alpha}$ in Theorem \ref{thm-christ}  is called  a dyadic cube of generation $k$ on $M$. 
From the property (iii) in Theorem \ref{thm-christ},  for any $(k,\alpha)$,  there is  a unique $\beta$ such that 
$Q^{k,\alpha}\subset Q^{k-1, \beta}$. We call $Q^{k-1,\beta}$ the predecessor of $Q^{k,\alpha}$. 
When $Q:=Q^{k,\alpha}$,  we denote the predecessor $Q^{k-1,\beta}$ by $\widetilde Q$ for simplicity.
\item For a given $R>0$,  we define $k_R\in\N$ to satisfy
\begin{equation*}
c_2\delta_0^{k_R-1}<R\leq c_2\delta_0^{k_R-2}.
\end{equation*}
  \end{enumerate}
\end{definition}
 The number $k_R$ means that a dyadic cube of generation $k_R$ is comparable to a ball of radius $R$.

 For the rest of the paper, we fix  some small numbers;
$$\delta:=\frac{2c_1}{c_2}\delta_0\in(0,\delta_0),\,\ \delta_1:= \frac{\delta_0(1-\delta_0)}{2}\in\left(0,\frac{\delta_0}{2}\right),
$$
$${\eta}:=\min(\delta,\delta_1)\in\left(0,\frac{1}{2}\right)\quad\mbox{and}\quad{\kappa}:= \frac{\eta}{2} \sqrt{1-\delta_0^2} . $$

By using the dyadic decomposition of a manifold $M$,  we have the following decomposition of $M\times (T_1,T_2]$ in space and time. For time variable, we take the standard euclidean dyadic decomposition.
\begin{lemma}[]\label{lem-decomp-x-t} 
There exists a countable collection $\{K^{k,\alpha}\subset M\times (T_1,T_2]: k\in\Z,\alpha\in J_k\}$ of   subsets of $M\times (T_1,T_2]\subset M\times\R$ and positive constants $0<\delta_0<1$, $c_1$ and $c_2$ (with $2c_1\leq c_2$)  that depend only on $n,$ such that
\begin{enumerate}[(i)]
\item $\left| M\times(T_1,T_2]\backslash \bigcup_{\alpha}K^{k,\alpha}\right|=0$ for $k\in\Z$,
\item if $l\leq k$, $\alpha\in J_k$, and $\beta\in J_l$, then either $K^{k,\alpha}\subset K^{l,\beta}$ or $K^{k,\alpha}\cap K^{l,\beta}=\emptyset$,
\item for any $(k,\alpha)$ and any $l<k$, there is a unique $\beta$ such that $K^{k,\alpha}\subset K^{l,\beta}$,
\item $\text{diam}(K^{k,\alpha})\leq c_2\delta^k_0\times c_2^2\delta_0^{2k}$,
\item any $K^{k,\alpha}$ contains some cylinder $B_{c_1\delta^k_0}(z^{k,\alpha})\times (t^{k,\alpha}-c_1^2\delta_0^{2k},t^{k,\alpha}]$.
\end{enumerate}
\end{lemma}
\begin{proof} To decompose in time variable, for each $k\in\Z$,  we select the largest  integer $N_k\in\Z$ to satisfy
$$ 
\frac{1}{4}c_2^2 \delta_0^{2k}\leq\frac{T_2-T_1}{2^{2N_k}}< c_2^2\delta_0^{2k}. $$  
For $k$-th generation, we split the interval $(T_1,T_2]$ into $2^{2N_k}$ disjoint  subintervals which have the same length. Then  we obtain $|J_k|=|I_k|\cdot 2^{2N_k}$ disjoint  subsets on $M\times(T_1,T_2]$ satisfying  properties  (i)-(v).
\end{proof}

For the rest of this section, let  $\{K^{k,\alpha}\subset M\times (T_1,T_2]: k\in\Z,\alpha\in J_k\}$ be the parabolic dyadic decomposition of $M\times(T_1,T_2]$ as in Lemma \ref{lem-decomp-x-t}.
\begin{definition}[Parabolic dyadic cubes ]\item
 \begin{enumerate}[(i)]
\item  $ K=K^{k,\alpha}$ is called  a parabolic dyadic cube of generation $k$. If $ K:=K^{k,\alpha}\subset K^{k-1,\beta}=:\widetilde{K}$, we say
$\widetilde{K}$ is the predecessor of $K$.
\item For a parabolic dyadic cube $K$ of generation $k$,  we define $l(k)$ to be the length of $K$ in time variable, namely, $l(k)=\frac{T_2-T_1}{2^{2N_k}}$ for $M\times(T_1,T_2]$ in Lemma \ref{lem-decomp-x-t}.
\end{enumerate}
\end{definition}

We quote the following technical lemma proven by Cabr\'e  \cite[Lemma 6.5]{Ca}.
\begin{lemma}[Cabr\'e]\label{lem-decomp-cube-ball}
 Let $z_o\in M$ and $R>0$. Then we have the following.
  \begin{enumerate}[(i)]
 \item  If $Q$ is a dyadic cube of generation $k$ such that
 $$k\geq k_R\quad\mbox{and}\quad Q\subset B_{R}(z_o), $$
 then there exist $z_1\in Q$ and $r_k\in(0,R/2)$ such that
 \begin{equation}\label{eq-decomp-m-1}B_{\delta r_k}(z_1)\subset Q\subset\widetilde Q\subset \overline{B_{2r_k}(z_1)}\subset B_{\frac{11}{\eta}r_k}(z_1)\subset B_{\frac{11}{\eta}R}(z_o)
 \end{equation}
 and 
 \begin{equation}\label{eq-decomp-m-2}
 B_{\frac{9}{\eta}R}(z_o)\subset B_{\frac{11}{\eta}R}(z_1).\end{equation} 
In fact, for $k\geq k_R,$  the above radius $r_k$ is defined  by $$r_k:=\frac{1}{2}c_2\delta_0^{k-1}=\frac{c_1}{\delta}\delta_0^k.$$ 
 \item If $Q$ is a dyadic cube of generation $k_R$ and $d(z_o, Q)\leq \delta_1R$,  then $Q\subset B_{R}(z_o) $ and hence \eqref{eq-decomp-m-1} and \eqref{eq-decomp-m-2} hold for some $z_1\in Q$ and $r_{k_R}\in \left[\frac{\delta_0 R}{2},\frac{R}{2}\right)$. Moreover, 
 $$B_{\delta_1R}(z_o)\subset B_{2r_{k_R}}(z_1). $$
 \item There exists at least one dyadic cube $Q$ of generation $k_R$ such that $d(z_o, Q)\leq\delta_1R$.  \end{enumerate}
\end{lemma}

We remark that for $k\geq k_R, $$$ \eta^2r_k^2\leq\delta_0^2r_k^2 =\frac{1}{4}c_2^2\delta_0^{2k}\leq l(k) <c_2^2\delta_0^{2k}= 4r_{k+1}^2 $$
and \eqref{eq-decomp-m-1} gives that for any $a\in\R$, 
  \begin{align}\label{eq-decomp-m-3}
  K_{\eta r_k}(z_1,a)\subset  Q\times\big( a- l(k),a\big]\subset   \overline{K_{2r_k}(z_1,a)}  
\end{align}

  \begin{definition} Let $m\in\N.$ 
 For any parabolic  dyadic cube $K:=Q\times(a-l(k),a]$ of generation $k$,   the elongation of $K$ along time in $m$ steps  (see \cite{KL}), denoted by $\overline{K}^m$, is defined by$$ \overline{K}^m:=\widetilde{Q}\times\big(a,a+m\cdot l(k-1)\big], $$
 where  $l(k)$ is the length of a parabolic dyadic  cube of generation $k$ in time and  $\widetilde{Q}$ is the predecessor of $Q$ in space. The elongation $\overline{K}^m$ is the union of the stacks of parabolic dyadic cubes congruent to the predecessor of $K$.  
  \end{definition}
 Now we have a parabolic version of Calder\'on-Zygmund lemma. The proof of lemma is the same as Euclidean case so we  refer to  \cite{W} for the proof.
 
\begin{lemma}[Lemma 3.23, \cite{W}]\label{lem-cz}
Let $K_1=Q_1\times\big(a-l(k_0),a\big]$ be a parabolic  dyadic cube of generation $k_0$
 in $M\times(T_1,T_2],$ 
and  let $0<\alpha<1$ and $m\in\N$. Let $\cA\subset K_1$ be a measurable set such that $|\cA\cap K_1|\leq \alpha|K_1|$ and let
 $$\cA^m_{\alpha}:=\cup\left\{\overline{K}^m:|K\cap \cA| > \alpha|K|,\, K,\,\text{a parabolic dyadic cube in}\,\,\,K_1\right\}\cap \left(Q_1\times\R\right). $$ 
Then, we have $$|\cA_{\alpha}^m|\geq \frac{m}{(m+1)\alpha}|\cA|. $$
\end{lemma}
 


\section{Harnack inequality}\label{sec-harnack}
In order to prove the parabolic Harnack  inequality, we take the approach presented in \cite{W} and iterate Lemma \ref{lem-decay-est-1-step} with Christ decomposition (Theorem \ref{thm-christ}) and Calder\'on-Zygmund type lemma (Lemma \ref{lem-cz}). 
We begin this section with recalling that $\eta\in\left(0,\frac{1}{2}\right)$ is fixed as in the previous section. So the uniform constants $\mu_\eta, \ve'_\eta$ and $M'_{\eta}$ in Proposition \ref{cor-decay-2}  are also fixed and we denote them by $\mu,\ve_0$ and $M_0$ for simplicity.

We select an integer  $m>1$ large enough to satisfy 
$$ \frac{m}{(m+1)(1-\mu)} > \frac{1}{1-\frac{\mu}{2}}, $$ where $0<\mu<1$ is the constant in Lemma \ref{lem-decay-est-1-step}. For $T_1:=-3R^2$ and $T_2:=\left(\frac{16}{\eta^2}+1+m\right)R^2 $,  we consider a parabolic dyadic decomposition of $M\times(T_1,T_2]$  in Lemma \ref{lem-decomp-x-t} and fix the decomposition for Section \ref{sec-harnack}.

\subsection{Power decay estimate of  super-level sets}\label{sec-decay}

\begin{lemma}\label{lem-decay-est}
Suppose that $M$ satisfies the conditions \eqref{cond-M-1},\eqref{cond-M-2}. Let $z_o\in M, R>0$ and  $\tau \in[3,16]. $    
Let $u$ be a nonnegative smooth function such that $\sL u\leq f $  in $
 B_{\frac{50}{\eta^3} R}(z_o)\times\left(-3R^2,\frac{\tau R^2}{\eta^2}\right]$.  
 Assume that  
$$\displaystyle\inf_{B_{R}(z_o)\times\left[\frac{2R^2}{\eta^2}, \frac{\tau R^2}{\eta^2} \right]}u\leq 1$$ and $$ R^2\left( \fint_{B_{\frac{50}{\eta^3} R}(z_o)\times\left(-3R^2,\frac{\tau R^2}{\eta^2}\right]}|f^+|^{n+1}\right)^{\frac{1}{n+1}}\leq {\ve_1} $$
 for a uniform constant $0<\ve_1<\ve_0.$ 
Let $K_1$  be a parabolic dyadic cube of generation $k_R$ such that
\[
K_1:=Q_1\times\left(t_1-l(k_R),t_1\right]\subset Q_1\times(-R^2,R^2),
\]
where  $Q_1$ is a dyadic cube of generation $k_R$ such that $d(z_o, Q_1)\leq \delta_1 R$.
Then for $i=1,2,\cdots$,  we have
\begin{equation}\label{eq-decay-est}
\frac{\left|\{u>M_1^i\}\cap K_1\right|}{|K_1|}<\left(1-\frac{\mu}{2}\right)^i,
\end{equation}
where $0<\ve_1<\ve_0$ and $M_1>0$ depend only on $n, \lambda, \Lambda$,  and $ a_L$.
\end{lemma}
\begin{proof} (i) As Proposition \ref{cor-decay-2}, we use  Lemma \ref{lem-ext-time} to assume that
 a nonnegative smooth function $u$ defined  on $ B_{\frac{49}{\eta^3} R}(z_o)\times\left(T_1,T_2\right]$ satisfies that $\displaystyle\inf_{B_{R}(z_o)\times\left[\frac{2R^2}{\eta^2}, \frac{\tau R^2}{\eta^2} \right]}u\leq 1$ and $\sL u\leq f$ in $ B_{\frac{49}{\eta^3} R}(z_o)\times\left(T_1,T_2\right]$ for some $f$ with 
    $$ R^2\left( \fint_{B_{\frac{49}{\eta^3} R}(z_o)\times\left(T_1,T_2\right]}|f^+|^{n+1}\right)^{\frac{1}{n+1}} \leq \frac{50}{49}\ve_1<2\ve_1. $$

(ii) According to Lemma \ref{lem-decomp-cube-ball},  there exists a dyadic cube $Q_1\subset B_R(z_o)$ of generation $k_R$  such that $d(z_o, Q_1)\leq\delta_1 R$.  We  find
  $z_1\in Q_1$ and $r_{k_R}\in [\frac{\delta_0}{2}R,\frac{1}{2}R)$ satisfying 
 \eqref{eq-decomp-m-1},\eqref{eq-decomp-m-2} and  $B_{\delta_1R}(z_o)\subset B_{2r_{k_R}}(z_1)$.  Since $ \eta^2r_{k_R}^2\leq l(k_R)<4r_{k_R+1}^2=4\delta_0^2r_{k_R}^2<\delta_0^2R^2$,   we find $t_1\in(-R^2+l(k_R), R^2)$  such that   $K_1:=Q_1\times(t_1-l(k_R), t_1]$ is a parabolic dyadic cube of generation $k_R$ of  $M\times(T_1,T_2]$.  From \eqref{eq-decomp-m-3}, we also have that
   \begin{align*}
  K_{\eta r_{k_R}}(z_1,t_1) &\subset K_1 \subset \overline {K_{2r_{k_R}}(z_1,t_1) }.
\end{align*}

We use the induction  to prove \eqref{eq-decay-est} so we first check  the case $i=1$.  
We notice that $d(z_o,z_1)<R$,   $r_{k_R}\in[\frac{\delta_0}{2}R,\frac{1}{2}R)\subset(\frac{\eta}{2}R,R)$ and $|t_1|< R^2$. We set $\ve_1:=\left(\frac{3/\eta^2+3}{16/\eta^2+m+4}\right)^{\frac{1}{n+1}}\frac{\ve_0}{2}.$  
Then,   $u$ satisfies the hypotheses of Proposition \ref{cor-decay-2}  with $r=r_{k_R}$  and $N=1$,    so we   deduce that
\begin{align*}
0 <\mu&\leq\frac{\left|\{u\leq M_0^3\}\cap K_{\eta r_{k_R}}(z_1,t_1) \right|}{|K_{\alpha_1r_{k_R},\,\alpha_2r_{k_R}^2}(z_1 , t_1+4r_{k_R}^2) |}
=\frac{\left|\{u\leq M_0^3\}\cap K_{\eta r_{k_R}}(z_1,t_1) \right|}{|K_{\alpha_1r_{k_R},\,\alpha_2r_{k_R}^2}(z_1 , t_1) |}<\frac{\left|\{u\leq M_0^3\}\cap K_1\right|}{|K_1|}.
\end{align*}
Thus,  we have for $M_1\geq M_0^3$,   $$\frac{\left|\{u >M_1\}\cap K_1\right|}{|K_1|}\leq1-\mu<1-\frac{\mu}{2}. $$ 

(iii) Now, suppose that \eqref{eq-decay-est} is true for $i$,  that is, \begin{equation*} 
\frac{\left|\{u>M_1^i\}\cap K_1\right|}{|K_1|}<\left(1-\frac{\mu}{2}\right)^i.
\end{equation*}
To show the (i+1)-th step, define for $h>0$,  $$\cB_h:=\{u>h\}\cap {\color{black}B_{\frac{49}{\eta^3} R}(z_o)\times(T_1,T_2]}. $$  We know  $ \displaystyle\frac{\left|\cB_{M_1^i}\cap K_1 \right|}{|K_1|} <\left(1-\frac{\mu}{2}\right)^i$.  If $h>0$ is a constant such that 
 $$\frac{|\cA|}{|K_1|}\geq \left(1-\frac{\mu}{2}\right)^{i+1}\quad\mbox{for}\quad\cA:= \cB_{hM_1^i}\cap K_1, $$ then we will show that $h< M_1$ for a uniform constant $M_1>M_0>1$,  that will be fixed later. 


   Suppose on the contrary that $h\geq M_1$.       
   From (ii),  we  have
    $\displaystyle\frac{|\cA|}{|K_1|}\leq \frac{|\cB_{M_0^3}\cap K_1|}{|K_1|} \leq 1-\mu$ for $M_1\geq M^3_0$  and    $h\geq1$.    
Applying Lemma~\ref{lem-cz}  to $\cA$ with $\alpha=1-\mu$, it follows that 
$$|\cA_{1-\mu}^m|\geq \frac{m}{(m+1)(1-\mu)}|\cA| > \frac{1}{1-\frac{\mu}{2}}|\cA|. $$
We claim that 
\begin{equation}\label{eq-decay-est-elongation}\cA_{1-\mu}^m\subset \cB_{\frac{hM_1^i}{M_0^m}}\end{equation} for  $h\geq C_1M_0^m>1$,  where a uniform constant $C_1>0$ will be chosen. If not, there is  a point $(x_1,s_1)\in \cA^m_{1-\mu}\bs \cB_{\frac{hM_1^i}{M_0^m}}$ and we find a parabolic  dyadic cube $K:=Q\times(a-l(k),a]\subset K_1$ of generation $k(> k_R)$ such that 
$$|\cA\cap K|>(1-\mu)|K|\quad\mbox{and}\quad (x_1,s_1)\in \overline{K}^m$$ from the definition of $\cA^m_{1-\mu}$.  According to Lemma \ref{lem-decomp-cube-ball}, there exist  $z_1\in Q\subset Q_1\subset B_R(z_o)$ and $r_k\in(0,R/2)$  satisfying \eqref{eq-decomp-m-1}, \eqref{eq-decomp-m-2}, $K_{\eta r_k}(z_1,a)\subset K \subset \overline {K_{2r_{k}}(z_1,a)}$ and $$(x_1,s_1)\in \overline K^m=\widetilde{Q}\times\big(a, a+m\cdot l(k-1)\big]\subset  \overline{B}_{2r_k}(z_1)\times(a,a+m\cdot 4r_k^2]. $$
We note that $$\displaystyle \inf_{B_{2r_k}(z_1)\times(a,a+m\cdot 4r_k^2]}u\leq u(x_1,s_1)\leq \frac{hM_1^i}{M_0^m}$$ and  $$B_{\alpha_1r_k}(z_1)\times  \left(a-(\eta^2+\eta^4/4)r_k^2, a+m\cdot 4r_k^2\right]\subset B_{\alpha_1R}(z_o)\times\left(-3R^2, (1+m)R^2\right],$$  
  since  $ r_k< R/2$ and $ a\in(t_1-l(k_R),t_1]\subset(-R^2,R^2)$. 
We also have that  for  $j=1,\cdots,m$,  
 \begin{equation}\label{eq-decay-est-f}
 r_k^2\,\left(\fint_{K_{\alpha_1r_k, \alpha_2r_k^2}(z_1,a+(m-j+1)\cdot 4r_k^2 )}|f^+|^{n+1}\right)^{\frac{1}{n+1}}  \leq \ve_0\frac{h M_1^i}{M_0^{m-j+1}}.
\end{equation}
Indeed, the volume comparison theorem and the property \eqref{eq-decomp-m-2} will give that 
\begin{align*}
& r_k^2\,\left(\fint_{K_{\alpha_1r_k, \alpha_2r_k^2}(z_1,a+(m-j+1)\cdot 4r_k^2 )}|f^+|^{n+1}\right)^{\frac{1}{n+1}}  
=\frac{r_k^{\frac{n}{n+1}}\cdot r_k^{\frac{n+2}{n+1}} }{\left|B_{\alpha_1r_k}(z_1)\right|^{\frac{1}{n+1}} \left(\alpha_2r_k^2\right)^{\frac{1}{n+1}}}||f^+||_{L^{n+1} }\\
&\leq \frac{R^{\frac{n}{n+1}}\cdot R^{\frac{n+2}{n+1}}}{\left|B_{\alpha_1R}(z_1)\right|^{\frac{1}{n+1}} \left(\alpha_2R^2\right)^{\frac{1}{n+1}}}||f^+||_{L^{n+1}\left(B_{\frac{12}{\eta}R}(z_o)\times(T_1,T_2]\right)}\\
&\leq \frac{R^2}{\left|B_{\beta_1R}(z_o)\right|^{\frac{1}{n+1}} \left(\alpha_2R^2\right)^{\frac{1}{n+1}}}||f^+||_{L^{n+1}\left(B_{\frac{49}{\eta^3} }(z_o)\times(T_1,T_2]\right)}\\
&\leq \frac{C_1R^2/2}{\left|B_{\frac{49}{\eta^3} }(z_o)\times(T_1,T_2]\right|^{\frac{1}{n+1}}}  ||f^+||_{L^{n+1}\left(B_{\frac{49}{\eta^3} }(z_o)\times(T_1,T_2]\right)}< C_1\ve_1, 
\end{align*}
where a uniform constant $C_1>1$ depends only on $\eta, n$ and $m$. For $h\geq C_1M_0^m$ and $M_1>1$,  we have that
$$ r_k^2\,\left(\fint_{K_{\alpha_1r_k, \alpha_2r_k^2}(z_1,\,a+(m-j+1)\cdot 4r_k^2 )}|f^+|^{n+1}\right)^{\frac{1}{n+1}} < C_1\ve_1< \ve_0\frac{h M_1^i}{M_0^m} \leq \ve_0\frac{h M_1^i}{M_0^{m-j+1}}, $$ 
 which proves \eqref{eq-decay-est-f}. 
Thus,  we can  apply Lemma \ref{lem-decay-est-1-step} iteratively to $\tilde u_j:=\frac{M_0^{m-j+1}}{hM_1^i}u,$  for $ 1\leq j\leq m$,   to deduce 
$$\mu\leq \frac{\left|\left\{u \leq hM_1^i\right\}\cap K_{\eta r_k}(z_1,a)\right|}{|K_{\alpha_1r_k,\,\alpha_2r_k^2}(z_1,a+4r_k^2)|}<\frac{\left|\left\{u\leq h M_1^i\right\}\cap K\right|}{|K|} . $$ However, this contradicts to the fact that $|\cA\cap K|>(1-\mu)|K|$. 
Therefore, we have proved that $\cA_{1-\mu}^m\subset \cB_{\frac{hM_1^i}{M_0^m}}$ for   $h\geq C_1M_0^m$.

(iv) Since  $ |\cB_{M_1^i}\cap K_1 |<\left( 1-\frac{\mu}{2}\right)^{i}|K_1|$, 
 we have that $ \displaystyle {\big|\cB_{\frac{hM_1^i}{M_0^m}}\cap K_1\big|} \leq  |\cB_{M_1^i}\cap K_1 |<\left( 1-\frac{\mu}{2}\right)^{i}|K_1| \leq\frac{1}{1-\frac{\mu}{2}}{|\cA|}$   for   $h\geq C_1M_0^m$.
 Then, by using \eqref{eq-decay-est-elongation}, we obtain
\begin{align*}
|\cA_{1-\mu}^m\bs K_1|&=|\cA_{1-\mu}^m|-|\cA_{1-\mu}^m\cap K_1|\\
&\geq \frac{m}{(m+1)(1-\mu)}|\cA|-{\left|\cB_{\frac{hM_1^i}{M_0^m}}\cap K_1\right|}\\
&>\left(\frac{m}{(m+1)(1-\mu)}-\frac{1}{1-\frac{\mu}{2}}\right)|\cA|=:\alpha|\cA|\geq \alpha\left(1-\frac{\mu}{2}\right)^{i+1}|K_1|
\end{align*}
with $\alpha:=\frac{m}{(m+1)(1-\mu)}-\frac{1}{1-\frac{\mu}{2}}>0$.
 We find a point $(x_1,s_1)\in \cA^m_{1-\mu}\bs K_1$  and a parabolic dyadic cube $K:=Q\times(a-l(k),a]\subset K_1$ of generation $k(>k_R)$ such that $(x_1,s_1)\in \overline{ K}^m, $ and $|\cA\cap K|>(1-\mu)|K|$. We  may assume that $$s_1>t_1+\frac{\alpha}{2}\left(1-\frac{\mu}{2}\right)^{i+1}l(k_R)$$ since   $\cA_{1-\mu}^m\subset Q_1\times(t_1-l(k_R),+\infty)$ and    $\frac{|\cA_{1-\mu}^m\bs K_1|}{|Q_1|}> \alpha\left(1-\frac{\mu}{2}\right)^{i+1}l(k_R)$.  
Using Lemma \ref{lem-decomp-cube-ball} again, there exist $z_1\in Q\subset Q_1\subset B_R(z_o)$ and $r_k
\in (0,R/2)$  satisfying \eqref{eq-decomp-m-1},\eqref{eq-decomp-m-2}, and  $K_{\eta r_k}(z_1,a)\subset K \subset \overline {K_{2r_{k}}(z_1,a)}$.
  Then   we have 
$$s_1 \leq  a +m\cdot l(k-1) <t_1+m\cdot4r_k^2$$ and hence 
$$r_k\geq \frac{\sqrt\alpha  }{\sqrt{8m}}\left(1-\frac{\mu}{2}\right)^{\frac{i+1}{2}}\sqrt{l(k_R)}\geq \frac{\sqrt\alpha \delta_0^2}{4\sqrt{2m}}\left(1-\frac{\mu}{2}\right)^{\frac{i+1}{2}}R\geq \left(\frac{\eta}{2}\right)^{Ni}R $$ 
for a uniform integer $N>0$ independent of $i\in\N$.  We apply  Proposition \ref{cor-decay-2} to 
$u$   in order to get
  \begin{align*}
\mu &\leq \frac{\left|\left\{u\leq M_0^{Ni+2} \right\}\cap K_{\eta r_k}(z_1,a)\right|}{\left|K_{\alpha_1r_k,\,\alpha_2r_k^2}(z_1,a+4r_k^2) \right|}\leq  \frac{\left|\left\{u\leq M_0^{(N+2)i} \right\}\cap K \right|}{\left|K \right|},
\end{align*}
since 
 $r_k\geq \left(\frac{\eta}{2}\right)^{Ni}R$,  and $ (z_1,a)\in  K\subset K_1  \subset  B_R(z_o)\times(-R^2,R^2)$.  
If $h\geq M_1:=\max\{C_1M_0^m, M_0^{N+2}\}$,   this implies
$$1-\mu>  \frac{\left|\left\{u>M_0^{(N+2)i} \right\}\cap K \right|}{\left|K \right|} \geq  \frac{\left|\left\{u>hM_1^i \right\}\cap K \right|}{\left|K \right|}=\frac{ |\cA\cap K|}{\left|K \right|}, $$
which is 
 a contradiction to the fact  that $|\cA\cap K|>(1-\mu)|K| $. Thus, we have  $h< M_1$ for a uniform constant $M_1:=\max\{ C_1 M_0^m, M_0^{N+2}\}.$ Therefore, we conclude that $\frac{\left|\{u>M_1^{i+1}\}\cap K_1\right|}{|K_1|}<\left(1-\frac{\mu}{2}\right)^{i+1}$,  completing the proof.   
\end{proof}
The following corollary is a direct consequence of Lemma \ref{lem-decay-est}, which estimates  the distribution function of $u$.
\begin{cor}\label{cor-decay-est}
Under  the same  assumption as  Lemma \ref{lem-decay-est}, we have
\begin{equation}\label{eq-decay-est-c}
\frac{\left|\{u \geq h\}\cap K_1\right|}{|K_1|}\leq d h^{-\e}\quad\forall h>0,
\end{equation}
where $d>0$ and $0<\e<1$  depend only on $n, \lambda, \Lambda$,  and  $a_L$.
\end{cor}
Another consequence of Lemma \ref{lem-decay-est} is a weak Harnack inequality for nonnegative supersolutions to $\sL u=f$.
\begin{cor}\label{cor-decay-est-2}
Under  the same  assumption as  Lemma \ref{lem-decay-est}, we have for $p_o:=\frac{\e}{2}$, 
\begin{equation}\label{eq-decay-est-c-2}
\left(\frac{1}{|K_{\kappa R}(z_o,0)|} \int_{K_{\kappa R}(z_o,0) }u^{p_o}\right)^{\frac{1}{p_o}} \leq C,
\end{equation}
where $C>0$ depends only on $n, \lambda, \Lambda$,  and  $a_L$.
 \end{cor}
\begin{proof} Let $k=k_R$ and let $\left\{K^{k,\alpha}:=Q^{k,\alpha}\times(t^{k,\alpha}-l(k), t^{k,\alpha}]\right\}_{\alpha\in J'_k}$ be a family of parabolic dyadic cubes intersecting $K_{\kappa R}(z_o,0)$. For $\alpha\in J_k'$,  we have that $K^{k,\al}\subset B_{(\kappa+\delta_0)R}(z_o)\times(-R^2,R^2]$ since $d(z_o, Q^{k,\alpha})\leq \kappa R (<\delta_1 R)$,    $diam(Q^{k,\alpha})\leq c_2\delta_0^k\leq \delta_0R$,  and $-R^2+l(k)<-\kappa^2R^2 \leq t^{k,\alpha}\leq l(k)< \delta_0^2R^2$.  Since 
\begin{align*}
|K_{\kappa R}(z_o,0)| 
&\geq \left(\frac{\kappa}{\kappa+\delta_0}\right)^n |B_{(\kappa+\delta_0)R}(z_o)|\cdot \kappa^2R^2\geq \left(\frac{\kappa}{\kappa+\delta_0}\right)^n \sum_{\alpha\in J'_k}|Q^{k,\alpha}|\cdot \kappa^2R^2\\ 
&\geq \left(\frac{\kappa}{\kappa+\delta_0}\right)^n \sum_{\alpha\in J'_k}|B_{c_1\delta_0^k}(z^{k,\alpha})|\cdot \kappa^2R^2\geq \left(\frac{\kappa}{\kappa+\delta_0}\right)^n \sum_{\alpha\in J'_k}|B_{\frac{\delta\delta_0}{2}R}(z^{k,\alpha})|\cdot \kappa^2R^2\\
&\geq \left(\frac{\kappa}{\kappa+\delta_0}\cdot\frac{\delta\delta_0}{2 (\delta_0+2\kappa)}\right)^n \sum_{\alpha\in J'_k}|B_{(\delta_0+2\kappa)R}(z^{k,\alpha})|\cdot \kappa^2R^2\\
&\geq \left(\frac{\kappa}{\kappa+\delta_0}\cdot\frac{\delta\delta_0}{2  (\delta_0+2\kappa)}\right)^n \sum_{\alpha\in J'_k}|B_{\kappa R}(z_o)|\cdot \kappa^2R^2,
\end{align*} the number $|J_k'|$ of parabolic dyadic cubes intersecting $K_{\kappa R}(z_o,0)$ is uniformly bounded. Thus for some $K^{k,\alpha}$ with $\alpha\in J_k'$,  we have 
\begin{align*}
\int_{K_{\kappa R}(z_o)} u^{p_o} &\leq |J'_k|\cdot\int_{K^{k,\alpha}}u^{p_o}\\&\leq |J'_k|\cdot\left\{|K^{k,\al}|+ p_o\int_1^{\infty} h^{p_o-1}|\{u\geq h\}\cap {K^{k,\alpha}}|dh \right\}\\
&\leq |J'_k|\cdot\left\{|K^{k,\al}|+ p_od|{K^{k,\alpha}}|\int_1^{\infty} h^{p_o-1 -\e}dh\right\}.  
\end{align*}from Corollary \ref{cor-decay-est}, where $d$ and $\e$ are the constants in Corollary \ref{cor-decay-est}.

By using the volume comparison theorem, 
we conclude that 
\begin{align*}
\frac{1}{|K_{\kappa R}(z_o,0)|} \int_{K_{\kappa R}(z_o,0) }u^{p_o}&\leq  C_0\frac{|{K^{k,\alpha}}|}{|K_{\kappa R}(z_o,0)|}\leq  C_0\left(\frac{\kappa+\delta_0}{\kappa}\right)^n\cdot \frac{\delta_0^2}{\kappa^2}
\end{align*}
for $C_0:=|J'_k|  \cdot \left\{1+ p_od\int_1^{\infty} h^{-1-\e/2}dh\right\}$  since $K^{k,\alpha}\subset B_{(\kappa+\delta_0)R}(z_o)\times(t^{k,\alpha}-\delta_0^2R^2,t^{k,\alpha}]$.
\end{proof}

\subsection{Proof of Harnack Inequality}
So far, we have dealt with  nonnegative supersolutions. Now, we consider a nonnegative solution $u$ of $\sL u = f$. We apply Corollary  \ref{cor-decay-est} as in \cite{Ca} (see also \cite{W}) to solutions $C_1 -C_2u$ for some constants $C_1$ and $C_2.$ 
\begin{lemma}\label{lem-sup-est-1} Suppose that $M$ satisfies the conditions \eqref{cond-M-1},\eqref{cond-M-2}. Let $z_o\in M,  R>0$ and  $\tau \in[3,16]. $  
Let $u$ be a nonnegative smooth function such that $\sL u= f $   in $
 B_{\frac{50}{\eta^3} R}(z_o)\times\left(-3R^2,\frac{\tau R^2}{\eta^2}\right]$.   Assume that  
$\displaystyle\inf_{B_{R}(z_o)\times\left[\frac{2R^2}{\eta^2}, \frac{\tau R^2}{\eta^2} \right]}u\leq 1$ and $$  R^2 \left(\fint_{ B_{\frac{50}{\eta^3} R}(z_o)\times\left(-3R^2,\frac{\tau R^2}{\eta^2}\right]} |f|^{n+1}\right)^{\frac{1}{n+1}} \leq \frac{\ve_1}{4}=:\ve$$
for a uniform constant 
 $0<\ve_1<1 $ as in Lemma \ref{lem-decay-est}.  

Then there exist constants $\sigma>0$ and $\tilde M_0>1$ depending on $n,\lambda,\Lambda$ and $ a_L$  such that for $\nu:=\frac{\tilde M_0}{\tilde M_0-1/2}>1$,    the following holds:

If $j\geq 1$ is an integer and $z_1\in M$ and $t_1\in\R$ satisfy $$d(z_o,z_1)\leq  {\kappa R},\quad |t_1|\leq \kappa^2 R^2$$ and 
\begin{equation*}u(z_1,t_1)\geq \nu^{j-1} \tilde M_0,\end{equation*} then  
\begin{enumerate}[(i)]
\item 
$K_{\frac{50}{\eta^3} L_j,\,\left(3+\frac{\tau}{\eta^2}\right)L_j^2  }(z_1,t_1)\subset  B_{\frac{50}{\eta^3} R}(z_o)\times\left(-3R^2,\frac{\tau R^2}{\eta^2}\right]$, 
\item $\displaystyle\sup_{K_{\frac{50}{\eta^3} L_j,\,\left(3+\frac{\tau}{\eta^2}\right)L_j^2  }(z_1,t_1)}u \geq \nu^j \tilde M_0$, 
\end{enumerate}
where 
 $\,L_j:=\displaystyle\sigma {\tilde M_0}^{-\frac{\e}{n+2}}\nu^{-\frac{j\e}{n+2}}R $ and $0<\e<1$  as   in Corollary \ref{cor-decay-est}. 
\end{lemma}
\begin{proof}
We select $\sigma>0$ and $\tilde M_0> 1$  large so that
$$\sigma> \frac{c_2}{c_1\delta_0}\left(2d2^{\e} \right)^{\frac{1}{n+2}}$$
and $$\sigma {\tilde M_0}^{-\frac{\e}{n+2}}+d\tilde M_0^{-\e}\leq \frac{\kappa}{4}, $$
where $d, \e, c_1,c_2$ and $\delta_0$ are the constants in Corollary \ref{cor-decay-est} and Theorem \ref{thm-christ}. 
Since $L_j\leq \frac{\kappa R}{4}<\frac{\eta R}{8},\, d(z_o,z_1)\leq\kappa R<R$ and $|t_1|\leq \kappa^2R^2< \frac{\eta^2 R^2}{4}$,  we have  
\begin{align*}
&B_{\frac{50}{\eta^3} L_j}(z_1)\times\left( t_1-\left(3+\frac{\tau}{\eta^2}\right)L_j^2 , t_1\right]\subset  B_{\frac{50}{\eta^3} R}(z_o)\times\left(-3R^2,\frac{\tau R^2}{\eta^2}\right],\end{align*} so   (i) is true.   

Now, suppose on the contrary  that $$\displaystyle\sup_{K_{\frac{50}{\eta^3}L_j,\,\left(3+\frac{\tau}{\eta^2}\right)L_j^2  }(z_1,t_1)}u < \nu^j \tilde M_0. $$
Let $k_j:=k_{L_j}\geq k_R$ with $L_j$ in Definition \ref{def-d-cube-m}. From Lemma \ref{lem-decomp-cube-ball}, there exists a dyadic  cube $Q_{L_j}$  of generation $k_j$ such that $d(z_1,Q_{L_j})\leq \delta_1 L_j$. We also find a parabolic dyadic cube $K_{L_j}$ of generation $k_j$ such that 
$$K_{L_j}\subset Q_{L_j}\times\left(t_1-\frac{\tau L_j^2}{\eta^2}-L_j^2, t_1-\frac{\tau L_j^2}{\eta^2}+L_j^2\right)$$ since $l(k_j)<\delta_0^2L_j^2$.  
Let $K_1$ be the unique predecessor of $K_{L_j}$ of generation $k_R$,  that is, $$K_{L_j}\subset K_1:= Q_1\times (a-l(k_R),a].$$ 
Then we have \begin{align*}d(z_o,Q_1)\leq d(z_o,Q_{L_j})&\leq d(z_o,z_1)+d(z_1,Q_{L_j})\leq \kappa R+ \delta_1L_j < 
\delta_1R\\&\mbox{and 
$(a-l(k_R),a]\subset (-R^2, R^2)$ }\end{align*}  since 
\begin{align*}&l(k_R)+|t_1|+\frac{\tau L_j^2}{\eta^2}+L_j^2\leq l(k_R)+|t_1|+\frac{16L_j^2}{\eta^2}+L_j^2\\&\leq \delta_0^2R^2+\kappa^2R^2+\left(\frac{16}{\eta^2}+1\right)\frac{\kappa^2}{16}R^2=
\left\{\delta_0^2+\left(\frac{16}{\eta^2}+17\right)\frac{\eta^4(1-\delta_0^2)}{64}\right\}R^2
< R^2. \end{align*}
Now, we apply Corollary \ref{cor-decay-est}  to $u$ with $K_1$  to obtain
\begin{equation} \label{eq-lem-sup-1}
 \left|\left\{u\geq \nu^{j}\frac{\tilde M_0}{2} \right\}\cap K_{L_j}\right| \leq  \left|\left\{u\geq \nu^{j}\frac{\tilde M_0}{2} \right\}\cap K_1\right| \leq d \left(\nu^{j}\frac{\tilde M_0}{2}\right)^{-\e} |K_1| .
\end{equation}

On the other hand, we consider the function $$w:=\frac{\nu \tilde M_0-u/\nu^{j-1}}{(\nu-1)\tilde M_0}, $$ which  is nonnegative and   satisfies $$\sL w=-\frac{f}{\nu^{j-1}(\nu-1)\tilde M_0} \,\,\,\,\mbox{in}\,\,\,\,K_{\frac{50}{\eta^3}L_j,\,\left(3+\frac{\tau }{\eta^2}\right)L_j^2  }(z_1,t_1)$$  from the assumption. We also have $w(z_1, t_1)\leq 1$ and 
$$\frac{|f|}{\nu^{j-1}(\nu-1)\tilde M_0}\leq\frac{|f|}{ (\nu-1)\tilde M_0}=\frac{2(\tilde M_0-1/2)|f|}{\tilde M_0}\leq 2|f|. $$ 
  By using the volume comparison theorem  with $L_j\leq \frac{\kappa}{4}R<\frac{\eta R}{8}$ and $B_{\frac{11}{\eta}\frac{4}{\eta^2}\frac{\eta R}{8}}(z_o)\subset B_{\frac{50}{\eta^3} \frac{\eta R}{8}}(z_1)$,  we get 
\begin{align*}
&L_j^2\left(\fint_{K_{\frac{50}{\eta^3}L_j,\,\left(3+\frac{\tau }{\eta^2}\right)L_j^2  }(z_1,t_1)}|2f|^{n+1} \right)^{\frac{1}{n+1}}
=\frac{2L_j^2}{|B_{\frac{50}{\eta^3}L_j}(z_1)|^{\frac{1}{n+1}}\left\{\left(3+\frac{\tau }{\eta^2}\right)L_j^2\right\} ^{\frac{1}{n+1}}}||f||_{L^{n+1}}\\
&\leq\frac{2(\eta R/8)^2}{|B_{\frac{50}{\eta^3}\cdot\frac{\eta R}{8}}(z_1)|^{\frac{1}{n+1}}\left\{\left(3+\frac{\tau }{\eta^2}\right)(\eta R/8)^2\right\} ^{\frac{1}{n+1}}}||f||_{L^{n+1}}\\
&\leq\frac{2(\eta R/8)^2}{|B_{\frac{11}{\eta}\frac{4}{\eta^2}\frac{\eta R}{8}}(z_o)|^{\frac{1}{n+1}}\left\{\left(3+\frac{\tau }{\eta^2}\right)(\eta R/8)^2\right\} ^{\frac{1}{n+1}}}||f||_{L^{n+1}}\\
&\leq\frac{2R^2}{|B_{\frac{11}{\eta}\frac{4}{\eta^2}R}(z_o)|^{\frac{1}{n+1}}\left\{\left(3+\frac{\tau }{\eta^2}\right)R^2\right\} ^{\frac{1}{n+1}}}||f||_{L^{n+1}\left( B_{\frac{50}{\eta^3} R}(z_o)\times\left(-3R^2,\frac{\tau R^2}{\eta^2}\right]\right)}\\&\leq2\left(\frac{50}{44}\right)^{\frac{n}{n+1}} \frac{\ve_1}{4}\leq\ve_1 . 
\end{align*} Applying Corollary \ref{cor-decay-est} to $w$ in $K_{L_j}$,  we deduce that 
$|\{w\geq \tilde M_0\}\cap K_{L_j}|\leq d \tilde M_0^{-\e}|K_{L_j}|$,  i.e.,
$$\left|\left\{u\leq\nu^j\frac{\tilde M_0}{2}\right\}\cap K_{L_j}\right|\leq d \tilde M_0^{-\e}|K_{L_j}|. $$
Putting together with \eqref{eq-lem-sup-1}, we obtain 
\begin{align*}
|K_{L_j}| \leq 2d2^{\e}\nu^{-j\e}{\tilde M_0}^{-\e}|K_1|\end{align*}
since $d\tilde M_0^{-\e}\leq \frac{\kappa}{2}<1/2$.
From Theorem \ref{thm-christ}, there is a point $z_*\in Q_{L_j}$ such that $B_{c_1\delta_0^{k_j}}(z_*)\subset Q_{L_j}\subset Q_1\subset \overline B_{c_2\delta_0^{k_R}}(z_*)$. Then we have  
\begin{align*} 
\left|B_{c_1\delta_0^{k_j}}(z_*)\right| \cdot c_1^2\delta_0^{2k_j}\leq \left|B_{c_1\delta_0^{k_j}}(z_*)\right| \cdot l(k_j)&\leq |K_{L_j}|\\&\leq 2d2^{\e}\nu^{-j\e}{\tilde M_0}^{-\e}|K_1|=2d2^{\e}\nu^{-j\e}{\tilde M_0}^{-\e}|Q_1|\cdot l(k_R)\\
&<2d2^{\e}\nu^{-j\e}{\tilde M_0}^{-\e}| \overline B_{c_2\delta_0^{k_R}}(z_*)|\cdot c_2^2\delta_0^{2k_R}\\
&\leq 2d2^{\e}\nu^{-j\e}{\tilde M_0}^{-\e}\left(\frac{c_2\delta_0^{k_R}}{c_1\delta_0^{k_j}}\right)^{n}\left|B_{c_1\delta_0^{k_j}}(z_*)\right| c_2^2\delta_0^{2k_R}
\end{align*}
from the volume comparison theorem. This means $$ \delta_0^{k_j}< \left(2d2^{\e} \right)^{\frac{1}{n+2}}\tilde M_0^{-\frac{\e}{n+2}}\nu^{-\frac{j\e}{n+2}}\frac{c_2}{c_1}\delta_0^{k_R}. $$ Since $c_2\delta_0^{k_R-1}<R\leq c_2\delta_0^{k_R-2}$,    
we deduce that 
\begin{align*}
L_j&\leq c_2\delta_0^{k_j-2}\leq\frac{c_2^2}{c_1\delta_0^2} \left(2d2^{\e} \right)^{\frac{1}{n+2}}\tilde M_0^{-\frac{\e}{n+2}}\nu^{-\frac{j\e}{n+2}} \delta_0^{k_R} \\&<\frac{c_2}{c_1\delta_0} \left(2d2^{\e} \right)^{\frac{1}{n+2}}\tilde M_0^{-\frac{\e}{n+2}}\nu^{-\frac{j\e}{n+2}}R<\sigma   \tilde M_0^{-\frac{\e}{n+2}}\nu^{-\frac{j\e}{n+2}}R=L_j,\end{align*}
in contradiction to the definition of $L_j$. Therefore, (ii) is true.
\end{proof}
Thus we deduce the following lemma from Lemma \ref{lem-sup-est-1}.
\begin{lemma}\label{lem-sup-est-2} Suppose that $M$ satisfies the conditions \eqref{cond-M-1},\eqref{cond-M-2}. Let $z_o\in M,  R>0$ and $\tau\in[3,16]$.
Let $u$ be a nonnegative smooth function such that $\sL u= f $   in $
 B_{\frac{50}{\eta^3} R}(z_o)\times\left(-3R^2,\frac{\tau R^2}{\eta^2}\right]$.   Assume that  
$$\displaystyle\inf_{B_{R}(z_o)\times\left[\frac{2R^2}{\eta^2}, \frac{\tau R^2}{\eta^2} \right]}u\leq 1$$ and $$ R^2\left(\fint_{B_{\frac{50}{\eta^3}R}(z_o)\times\left(-3R^2,\frac{\tau R^2}{\eta^2}\right]}|f|^{n+1}\right)^{\frac{1}{n+1}} \leq \ve$$
for a uniform constant   $0<\ve<1 $ in Lemma \ref{lem-sup-est-1}.  Then 
$$\sup_{B_{\frac{\kappa R}{2}}(z_o)\times \left(-\frac{\kappa^2R^2}{4} ,\frac{\kappa^2R^2}{4}\right)} u\leq C, $$
where $C>0$ depends only on $n,\lambda,\Lambda$ and $a_L$.
\end{lemma}
\begin{proof}
We take $j_o\in\N$ such that 
$$\sum_{j=j_o}^{\infty} \frac{50}{\eta^3}L_j <\frac{\kappa R}{2}\quad\mbox{and}\quad\sum_{j=j_o}^{\infty}\left(3+\frac{16}{\eta^2}\right)L_j ^2<\frac{\kappa^2 R^2}{4}. $$  We claim that  $\displaystyle\sup_{B_{\frac{\kappa R}{2}}(z_o)\times \left(-\frac{\kappa^2R^2}{4} ,\frac{\kappa^2R^2}{4}\right)} u\leq \nu^{j_o-1}\tilde M_0$ with $\tilde M_0>1$  as in Lemma \ref{lem-sup-est-1}. 
 If it does not hold, then there is a point 
$(z_{j_o},t_{j_o})\in B_{\frac{\kappa R}{2}}(z_o)\times \left(-\frac{\kappa^2R^2}{4} ,\frac{\kappa^2R^2}{4}\right)$ such that $u(z_{j_o},t_{j_o})> \nu^{j_o-1}\tilde M_0$. Applying Lemma \ref{lem-sup-est-1} with $(z_1,t_1)=(z_{j_o},t_{j_o})$,   we can  find  a point $(z_{j_o+1},t_{j_o+1})\in K_{\frac{50}{\eta^3}L_j,\,\left(3+\frac{\tau }{\eta^2}\right)L_j^2  }(z_{j_o},t_{j_o}) $ such that $$u(z_{j_o+1},t_{j_o+1})\geq \nu^{j_o}\tilde M_0 . $$  
According to the choice of $j_o$,   we have  $$d(z_o,z_{j_o+1})\leq d(z_o,z_{j_o})+d(z_{j_o},z_{j_o+1}) <\frac{\kappa R}{2}+\frac{\kappa R}{2}={\kappa R}$$ and $$|t_{j_o+1}|\leq |t_{j_o}|+|t_{j_o}-t_{j_o+1}|< \frac{\kappa^2R^2}{4}+\frac{\kappa^2R^2}{4}<{\kappa^2R^2}. $$ Thus we iterate this argument to obtain a sequence of points $(z_{j},t_{j})$ for $j\geq j_o$ satisfying 
$$d(z_o,z_j)\leq{\kappa R},\,\,|t_j|\leq{\kappa^2 R^2}\quad\mbox{and}\quad u(z_{j},t_{j})\geq \nu^{j-1}\tilde M_0, $$
since  $\displaystyle d(z_o,z_{j})\leq d(z_o,z_{j_o})+\sum_{i=j_o}^{\infty}d(z_i,z_{i+1})\leq \frac{\kappa R}{2}+\sum_{i=j_o}^{\infty}\frac{50}{\eta^3}L_i<{\kappa R}$ and  $\displaystyle|t_i|\leq |t_{j_o}|+\sum_{i=j_o}^{\infty}|t_{i}-t_{i+1}|\leq \frac{\kappa^2R^2}{4}+\sum_{i=j_o}^{\infty}\left(3+\frac{\tau }{\eta^2}\right)L_i^2<{\kappa^2R^2}$ for $j\geq j_o$.  This contradicts to the continuity of $u$ and 
therefore  we conclude that $$\displaystyle\sup_{B_{\frac{\kappa R}{2}}(z_o)\times \left(-\frac{\kappa^2R^2}{4} ,\frac{\kappa^2R^2}{4}\right)} u\leq \nu^{j_o-1}\tilde M_0. $$ 
\end{proof}

Now the Harnack inequality  in Theorem \ref{thm-harnack} follows easily from Lemma \ref{lem-sup-est-2} by using a standard covering argument and the volume comparison theorem. 

\begin{thm}[Harnack Inequality]\label{thm-harnack}Suppose that $M$ satisfies the conditions \eqref{cond-M-1},\eqref{cond-M-2}. Let $z_o\in M, $ and $R>0$. 
Let $u$ be a nonnegative smooth function in $K_{2R}(0,4R^2)\subset M\times\R$. Then 
\begin{equation}
\sup_{K_R(z_o,2R^2)}u\leq C\left\{ \inf_{K_R(z_o, 4R^2)}u+ R^2\left(\fint_{K_{2R}(z_o,4R^2)}|\sL u|^{n+1} \right)^{\frac{1}{n+1}} \right\},
\end{equation} 
where  $C>0$ is a constant depending only on $n,\lambda,\Lambda$ and $a_L$.
\end{thm}

\begin{proof}
According to Lemma \ref{lem-sup-est-2}, for $\tau \in[3,16]$,  a nonnegative  smooth  function $v$ in $K_{\frac{50}{\eta^3}r,\,\left(3+\frac{\tau }{\eta^2}\right)r^2  }\left(\overline x, \overline t+\frac{\tau r^2}{\eta^2}\right)$ satisfies 
\begin{equation}\label{eq-harnack-1}
\sup_{K_{\frac{\kappa r}{2}}(\overline x, \overline t)}v\leq C\left\{\inf_{K_{\frac{\kappa r}{2} }(\overline x, \overline t+\frac{\tau r^2}{\eta^2})} v+ r^2\left(\fint_{K_{\frac{50}{\eta^3}r,\,\left(3+\frac{\tau}{\eta^2}\right)r^2  }\left(\overline x, \overline t+ \frac{\tau r^2}{\eta^2}\right)}|\sL v|^{n+1}\right)^{\frac{1}{n+1}}\right\}\end{equation}
 since $\frac{\kappa}{2}<1 $.  
 
 Now, let $(x,t)\in K_{R}(z_o,2R^2)=B_R(z_o)\times(R^2,2R^2]$ and $(y,s)\in K_R(z_o,4R^2)=B_R(z_o)\times(3R^2,4R^2]$. We show that 
\begin{equation*}
u(x,t)\leq C \left\{ u(y,s) +\frac{R^2}{|K_{2R}(z_o,4R^2)|^{\frac{1}{n+1}}}||\sL u ||_{L^{n+1}(K_{2R}(z_o,4R^2) )} \right\}
\end{equation*}
for a uniform constant $C>0$ depending only on $n,\lambda,\Lambda$ and $a_L$. 
 We consider a piecewise $C^1$ path $\gamma : [0,l]\to M,\,\,\gamma(0)=x, \gamma(l)=y,\,\,l < 2R$,  consisting of a minimal geodesic parametrized by arc length joining $x$ and $z_o$,  followed by a minimal geodesic parametrized by arc length joining $z_o$ and $y$.  We notice that $\gamma([0,l])\subset B_R(z_o)$ and  $d(\gamma(s_1),\gamma(s_2))\leq|s_1-s_2|$.

We can select  uniform constants  $A>0$ and $N\in\N$ such that $$A:=\max\left\{\frac{64}{3\kappa\eta^2}, \frac{50}{\eta^3}\right\}\,\,\,\,\,\mbox{and}\,\,\,\,\,\frac{3}{16}\eta^2A^2\leq N\leq \frac{1}{3}\eta^2A^2$$  since $\frac{16-9}{16\cdot 3}\eta^2A^2\geq \frac{7\cdot4}{3^2\kappa}A>1$.  For $i=0,1,\cdots, N, $ we define $$ (x_i, t_i):= \left(\gamma\left(i\frac{l}{N}\right), i \,\frac{s-t}{N} +t\right)\in B_R(z_o)\times [R^2, 4R^2]. $$ Then we have $(x_0,t_0)=(x,t), (x_N,t_N)= (y,s)$ and  for $i=0,\cdots, N-1$, 
 \begin{align*}
 d(x_{i+1},x_i)\leq \frac{l}{N} <\frac{2R}{N}\leq\frac{64}{3\kappa\eta^2A}\cdot\frac{\kappa R}{2A}  \leq \frac{\kappa}{2}\frac{R}{A},\\
\frac{3R^2}{\eta^2A^2}\leq\frac{R^2}{N}\leq  t_{i+1}-t_i=\frac{s-t}{N} \leq  \frac{3R^2}{N}\leq \frac{16R^2}{\eta^2A^2}.
 \end{align*}
We also have that $  K_{\frac{50}{\eta^3}\frac{R}{A},\, 3\frac{R^2}{A^2}+t_i-t_{i-1} }(x_i,t_i)\subset K_{2R}(z_o,4R^2)  $  for  $i=1,\cdots,N$ since $\frac{50}{\eta^3}\frac{R}{A}\leq R$.  We apply the estimate \eqref{eq-harnack-1} with $r=\frac{R}{A},\tau=(t_{i+1}-t_i)\frac{\eta^2A^2}{R^2} $  and $(\overline x, \overline t)=(x_{i+1},t_{i+1})$   for $i=0,1,\cdots, N-1$  and use the volume comparison theorem to have 
\begin{align*}
u(x_i,t_i)&\leq C \left\{ u(x_{i+1},t_{i+1}) +\frac{(R/A)^2}{|K_{\frac{50}{\eta^3}\frac{R}{A},\, 3\frac{R^2}{A^2}+t_i-t_{i-1} }(x_{i+1},t_{i+1})|^{\frac{1}{n+1}}}||\sL u ||_{L^{n+1}(K_{2R}(z_o,4R^2) )} \right\}\\
&\leq C \left\{ u(x_{i+1},t_{i+1}) +\frac{(R/A)^2}{|B_{\frac{50}{\eta^3}\frac{R}{A}}(x_{i+1})\cdot\left(3+\frac{3}{\eta^2}\right)\frac{R^2}{A^2}|^{\frac{1}{n+1}}}||\sL u||_{L^{n+1}(K_{2R}(z_o,4R^2) )} \right\}\\
&\leq C \left\{ u(x_{i+1},t_{i+1}) +\frac{R^2}{|B_{3R}(x_{i+1})\cdot 4R^2|^{\frac{1}{n+1}}}||\sL u||_{L^{n+1}(K_{2R}(z_o,4R^2) )} \right\},
\end{align*}where a uniform constant $C>0$ may change from line to line. 
Since  $B_{3R}(x_{i+1})\supset B_{2R}(z_o)$,  we deduce that 
\begin{align*}
u(x_i,t_i)
&\leq C \left\{ u(x_{i+1},t_{i+1}) +\frac{R^2}{|B_{2R}(z_o)\cdot 4R^2|^{\frac{1}{n+1}}}||\sL u||_{L^{n+1}(K_{2R}(z_o,4R^2) )} \right\}.
\end{align*}
Therefore,   we conclude that  \begin{equation*}
u(x,t)\leq C \left\{ u(y,s) +\frac{R^2}{|K_{2R}(z_o,4R^2)|^{\frac{1}{n+1}}}||\sL u||_{L^{n+1}(K_{2R}(z_o,4R^2) )} \right\}
\end{equation*}
for a uniform constant $C>0$ since $N\in\N$ is uniform.
\end{proof}

Arguing in a similar way as  Theorem \ref{thm-harnack},  Corollary \ref{cor-decay-est-2} gives the following   weak Harnack inequality. 
 \begin{thm}[Weak Harnack Inequality]\label{lem-weak-harnack} Suppose that $M$ satisfies the conditions \eqref{cond-M-1},\eqref{cond-M-2}. Let $z_o\in M, $ and $R>0$.
Let $u$ be a nonnegative smooth function such that $\sL u\leq f$ in $K_{2R} (z_o,4R^2)$. Then
\begin{equation*}
\left(  \fint_{K_{R}(z_o,2R^2)}u^{p_o}\right)^{\frac{1}{p_o}}\leq C\left\{ \inf_{K_{R}(z_o,4R^2)}u+R^2\left(\fint_{K_{2R}(z_o,4R^2)}|f^+|^{n+1} \right)^{\frac{1}{n+1}} \right\},
\end{equation*} 
where 
$0<p_o<1$ and $C>0$  depend  only on $n,\lambda,\Lambda$ and $a_L$.
\end{thm}
\begin{proof}
Let $\e>0$ be the constant in Corollary \ref{cor-decay-est} and let $p_o:=\frac{\e}{2}$. 
 We consider a parabolic decomposition of $M\times(0, 4R^2]$ according to Lemma \ref{lem-decomp-x-t}.  Let $k:= k_{\frac{\kappa R}{A}}$  for the  constant $A>0$ in the proof  of Theorem \ref{thm-harnack}.  Let $\left\{K^{k,\alpha}:=Q^{k,\alpha}\times(t^{k,\alpha}-l(k), t^{k,\alpha}]\right\}_{\alpha\in J'_k}$ be a family of parabolic dyadic cubes intersecting $K_{R}(z_o,2R^2)$. We note that  $diam(Q^{k,\alpha})\leq c_2\delta_0^k\leq \delta_0\cdot\frac{\kappa R}{A}$ and $l(k)\leq \delta_0^2\cdot\frac{\kappa^2 R^2}{A^2}$.  Following the same argument as Corollary \ref{cor-decay-est-2}, we deduce  that $|J'_k|$ is uniformly bounded and  
\begin{align*}
\int_{K_{R}(z_o,2R^2)} u^{p_o}\leq |J'_k|\int_{K^{k,\alpha}} u^{p_o} 
\end{align*}
for some $K^{k,\alpha}$ with $\alpha\in J'_k$. Then we find $(x,t)\in  K^{k,\alpha}\cap B_R(z_o)\times[R^2,(2+\delta_0^2\frac{\kappa^2}{A^2})R^2]$ such that  
$K^{k,\alpha}\subset K_{\frac{\kappa R}{A}}(x,t)$ since $diam(Q^{k,\alpha})\leq\delta_0\cdot\frac{\kappa R}{A}$ and $l(k)\leq\delta_0^2\cdot\frac{\kappa^2 R^2}{A^2}$.
Since $d(z_o,x)\leq R$ and $B_{\frac{\kappa R}{A}}(x)\subset B_{\left(1+\frac{\kappa}{A}\right)R}(z_o)$,   we have\begin{equation}\label{eq-w-hn}
\frac{1}{|K_{R}(z_o,2R^2)|}\int_{K_{R}(z_o,2R^2)}u^{p_o} \leq \frac{C_0}{|K_{\frac{\kappa R}{A}}(x,t)|}  \int_{K_{\frac{\kappa R}{A}}(x,t)}u^{p_o}
\end{equation}
for $C_0:=|J'_k|\left(1+\frac{\kappa}{A}\right)^n\cdot\frac{\kappa^2}{A^2}$ by using the volume comparison theorem.

We set 
$$ \inf_{K_R(z_o,4R^2)} u=:u(y,s)$$
for some $(y,s)\in\overline{K_R(z_o,4R^2)} $. As in the proof of Theorem \ref{thm-harnack} we take a piecewise geodesic path $\gamma$ connecting $x$  to $y$. Let $N\in\N$ be the constant in Theorem  \ref{thm-harnack}.  For $i=0,1,\cdots, N, $ we define $$ (x_i, t_i):= \left(\gamma\left(i\frac{l}{N}\right), i \,\frac{s-t}{N} +t\right)\in B_R(z_o)\times [R^2, 4R^2]. $$ Then we have $(x_0,t_0)=(x,t), (x_N,t_N)= (y,s)$ and  for $i=0,\cdots, N-1$, 
 \begin{align*}
 d(x_{i+1},x_i) <  \frac{\kappa}{2}\cdot\frac{R}{A}\quad\mbox{and }\quad
\frac{3R^2}{\eta^2A^2}\leq   t_{i+1}-t_i \leq \frac{16R^2}{\eta^2A^2}.
 \end{align*}
 It is easy to check that  for any $i=0,1,\cdots, N-1$,  $B_{\frac{\kappa R}{A}}(x_i)\cap B_{\frac{\kappa R}{A}}(x_{i+1})\supset B_{\frac{\kappa R}{2A}}(x_{i+1})$ and hence 
 \begin{equation}\label{eq-w-harnck-x}
 \begin{split}
\inf_{K_{\frac{\kappa R}{A}}(x_i,t_{i+1})} u\leq \inf_{K_{\frac{\kappa R}{2A}}(x_{i+1},t_{i+1})}u
&\leq \left\{\frac{1}{|K_{\frac{\kappa R}{2A}}(x_{i+1},t_{i+1})|}\int_{K_{\frac{\kappa R}{2A}}(x_{i+1},t_{i+1})} u^{p_o}\right\}^{\frac{1}{p_o}}\\&\leq 2^{\frac{n+2}{p_o}} \left\{\frac{1}{|K_{\frac{\kappa R}{A}}(x_{i+1},t_{i+1})| }\int_{K_{\frac{\kappa R}{A}}(x_{i+1},t_{i+1})} u^{p_o}\right\}^{\frac{1}{p_o}}.
\end{split}
\end{equation}

On the other hand, 
Corollary \ref{cor-decay-est-2} says  that for $i=0,1,\cdots, N-1$
\begin{align*}
&\left \{\frac{1}{|K_{\frac{\kappa R}{A}}(x_i,t_i)| }\int_{K_{\frac{\kappa R}{A}}(x_i,t_i)} u^{p_o}\right\}^{1/{p_o}} \\
&\leq C \left\{ \inf_{K_{\frac{\kappa R}{A}}(x_i,t_{i+1}) }u+\frac{(R/A)^2}{\left|K_{\frac{50}{\eta^3}\cdot \frac{R}{A},\, 3\frac{R^2}{A^2}+t_{i+1}-t_i  }\left(x_i,t_{i+1}\right)\right|^{\frac{1}{n+1}}}\left\|f^+\right\|_{L^{n+1}(K_{2R}(z_o,4R^2))}\right\}\\
&\leq C \left\{ \inf_{K_{\frac{\kappa R}{A}}(x_i,t_{i+1}) }u+\frac{R^2}{\left|K_{2R}\left( z_o,4R^2\right)\right|^{\frac{1}{n+1}}}\left\|f^+\right\|_{L^{n+1}(K_{2R}(z_o,4R^2))}\right\}
\end{align*} by using the same argument as Theorem \ref{thm-harnack} with the volume comparison theorem.  
 Combining with \eqref{eq-w-harnck-x}, we deduce  
\begin{align*} 
&\left \{\frac{1}{|K_{\frac{\kappa R}{A}}(x,t)| }\int_{K_{\frac{\kappa R}{A}}(x,t)} u^{p_o}\right\}^{1/{p_o}} \\
 &\leq C \left\{ \inf_{K_{\frac{\kappa R}{A}}(x_{N-1},t_{N}) }u+\frac{R^2}{\left|K_{2R}\left( z_o,4R^2\right)\right|^{\frac{1}{n+1}}}\left\|f^+\right\|_{L^{n+1}(K_{2R}(z_o,4R^2))}\right\}\\&\leq  C \left\{ \inf_{K_{\frac{\kappa R}{2A}}(x_{N},t_{N}) }u+\frac{R^2}{\left|K_{2R}\left( z_o,4R^2\right)\right|^{\frac{1}{n+1}}}\left\|f^+\right\|_{L^{n+1}(K_{2R}(z_o,4R^2))}\right\}\\
&\leq  C \left\{  u(y,s)+\frac{R^2}{\left|K_{2R}\left( z_o,4R^2\right)\right|^{\frac{1}{n+1}}}\left\|f^+\right\|_{L^{n+1}(K_{2R}(z_o,4R^2))}\right\},
\end{align*} for a uniform constant $C>0$ since $N\in\N$ is uniform.
Therefore, the result follows  from \eqref{eq-w-hn}.
\end{proof}

\begin{acknowledgment}
Seick Kim is supported by NRF Grant No. 2012-040411
 and R31-10049 (WCU program). Ki-Ahm Lee was supported by Basic Science Research Program through the National Research Foundation of Korea(NRF) grant funded by the Korea government(MEST)(2010-0001985).
\end{acknowledgment}


\end{document}